\documentclass[12pt,a4paper]{amsart}
\usepackage{amsfonts}
\usepackage{amsthm}
\usepackage{amsmath}
\usepackage{xcolor} 
\usepackage{amssymb}
\usepackage{mathtools} \usepackage{amscd}
\usepackage{dsfont} \usepackage[latin2]{inputenc}
\usepackage{t1enc}
\usepackage[mathscr]{eucal}
\usepackage{indentfirst}
\usepackage{graphicx}
\usepackage{graphics}
\PassOptionsToPackage{hyphens}{url}
\usepackage{hyperref}
\usepackage{epic}
\numberwithin{equation}{section}
\usepackage[margin=2.2cm, marginparwidth=2.4cm]{geometry}
\usepackage{epstopdf} 
\usepackage{enumitem}
\usepackage[sort,numbers]{natbib}

\usepackage{comment}
\numberwithin{equation}{section}
\newtheorem{assum}{Assumption}[section]
\newtheorem{theorem}{Theorem}[section]
\newtheorem{corollary}{Corollary}[section]
\newtheorem{definition}{Definition}[section]
\newtheorem{proposition}{Proposition}[section]
\newtheorem{remark}{Remark}[section]
\newtheorem{lemma}{Lemma}[section]

\DeclareMathOperator*{\argmin}{arg\,min}

\def\mF{{\mathcal F}}
\def\mP{{\mathcal P}}

\def\mB{{\mathcal B}}
\def\mE{{\mathcal E}}

\def\mM{{\mathcal M}}
\def\mI{{\mathcal I}}

\def\mS{{\mathcal S}}
\def\mG{{\mathcal G}}
\def\mN{{\mathcal N}}

\def\relu{\mathrm{ReLU}}
\def\sp{\mathrm{SP}}

\def\R{{\mathbb R}}
\def\Z{{\mathbb Z}}
\def\mZ{{\mathcal Z}}

\def\N{{\mathbb N}}

\def\bE{{\mathbf E}}
\def\gt{{\rightarrow}}

\def\eps{{\varepsilon}}

 \def\1{{\mathbf 1}}

\usepackage{todonotes}
\newcommand{\jl}[1]{\todo[color=orange!70]{JL: #1}}

\allowdisplaybreaks

\begin{document}

\title[A Priori Generalization Analysis of the Deep Ritz Method]{A Priori Generalization Analysis of the Deep Ritz Method for Solving High Dimensional Elliptic Equations}

\author{Jianfeng Lu}
\address{(JL) Departments of Mathematics, Physics, and Chemistry, Duke University, Box 90320, Durham, NC 27708.}
\email{jianfeng@math.duke.edu}
\author{Yulong Lu}
\address{(YL) Department of Mathematics and Statistics, Lederle Graduate Research Tower, University of Massachusetts, 710 N. Pleasant Street, Amherst, MA 01003.} 
\email{lu@math.umass.edu} 
\author{Min Wang}
\address{(MW) Mathematics Department, Duke University, Box 90320, Durham, NC 27708.}
\email{wangmin@math.duke.edu}

\thanks{J.L.~and M.W.~are supported in part by National Science Foundation via grants DMS-2012286 and CCF-1934964. Y.L.~is supported by the start-up fund of the Department of Mathematics and Statistics at UMass Amherst.}

\date{\today}

\begin{abstract}
This paper concerns the a priori generalization analysis of the Deep Ritz Method (DRM) [W. E and B. Yu, 2017], a popular neural-network-based method for solving high dimensional partial differential equations.  We derive the generalization error bounds  of two-layer neural networks in the framework of the DRM for solving two prototype elliptic PDEs: Poisson equation and static Schr\"odinger equation on the $d$-dimensional unit hypercube. Specifically,  we prove that the convergence rates of generalization errors are independent of the dimension $d$, under the a priori assumption that the exact solutions of the PDEs lie in a suitable low-complexity space called spectral Barron space. Moreover, we give sufficient conditions on the forcing term and the potential function which guarantee that the solutions are spectral Barron functions. We achieve this by developing a new solution theory for the PDEs on the spectral Barron space, which can be viewed as an analog of the classical Sobolev regularity theory for PDEs. 
\end{abstract}

\maketitle

\section{Introduction}

Numerical solutions to high dimensional partial differential equations (PDEs) have been a long-standing challenge in scientific computing. The impressive advance of deep learning has offered exciting possibilities for algorithmic innovations. In particular, it is a natural idea to represent  solutions of PDEs by (deep) neural networks to exploit the rich expressiveness of neural networks representation. The parameters of neural networks are then trained by optimizing some loss functions associated with the PDE. Natural loss functions can be designed using the variational structure, similar to the Ritz-Galerkin method in classical numerical analysis of PDEs. Such method is known as the Deep Ritz Method (DRM) in \cite{weinan2018deep, khoo2019solving}.
Methods in a similar spirit has been also developed in the computational physics literature \cite{carleo2017solving}  for solving eigenvalue problems arising from many-body quantum mechanics, under the framework of variational Monte Carlo method \cite{mcmillan1965ground}.

Despite wide popularity and many successful applications of the DRM and other approaches of using neural networks to solve high-dimensional PDEs, the analysis of such methods is scarce and still not well understood. 
This paper aims to provide an a priori generalization error analysis of the DRM with dimension-explicit estimates. 

\smallskip 

Generally speaking, the error of using neural networks to solve high dimensional PDEs can be decomposed into the following parts:
\begin{itemize}[wide]
\item Approximation error: this is the error of approximating the solution of a PDE using neural networks; 
\item Generalization error:  this refers to the error of the neural network-based approximate
solution on predicting unseen data. The variational problem involves integrals in high dimension, which can be expensive to compute. In practice Monte Carlo methods are usually used to approximate those high dimensional integrals and thus the miminizer of the surrogate model (known as empirical risk minimization) would be different from the minimizer of the original variational problem; 

\item Training (or optimization) error: this is the error incurred by the optimization algorithm used in the
training of neural networks for PDEs. Since the parameters of the neural networks are obtained through an optimization process, it might not be able to find the best approximation to the unknown solution within the function class. 
\end{itemize}

Note that from a numerical analysis point of view, these errors already appear for conventional Galerkin methods. Indeed, taking finite element methods for example, the approximation error is the error of approximating the true solution in the finite element space; the generalization error can be seen as the discretization error caused by numerical quadrature of the variational formulation; the optimization error corresponds to the computational error in the
conventional numerical PDEs due to the inaccurate resolution of linear or nonlinear finite
dimensional discrete system.

Although classical numerical analysis for PDEs in low dimensions has formed a relatively complete theory in the last several decades, the error analysis of neural network methods is much more challenging for high dimensional PDEs and requires new ideas and tools. In fact, the three components of error analysis highlighted above all face new difficulties. 

For approximation, as is well known, high dimensional problems suffer from the curse of dimensionality, if we proceed with standard regularity-based function spaces such as Sobolev spaces or H\"older spaces as in conventional numerical analysis. In fact, even using deep neural networks, the approximation rate for functions in such spaces deteriorate as the dimension becomes higher; see \cite{yarotsky2017error,yarotsky2018optimal}. 
Therefore, to obtain better approximation rates that scale mildly in the large dimensionality, it is natural to assume that the function of interest lies in a suitable smaller function space which has low complexity compared to Sobolev  or H\"older spaces so that the function can be efficiently approximated by neural networks in high dimensions. The first function class of this kind is the so-called {\em Barron space} defined in the seminal work of Barron  \cite{barron1993universal}; see also \cite{klusowski2018approximation,e2019barron,siegel2020approximation,siegel2020high} for more variants of Barron spaces and their neural-network  approximation properties. In the present paper we  will introduce 
a discrete version of Barron's definition of such space using the idea of spectral decomposition and because of this we adopt the terminology of {\em spectral Barron space} following \cite{siegel2020high,ma2020towards} to distinguish it from the other versions.
As the Barron spaces are very different from the usual Sobolev spaces, for PDE problems, one has to develop novel \emph{a priori} estimates and correspondingly approximation error analysis. In particular, a new solution theory for high dimensional PDEs in those low-complexity function spaces needs to be developed. This paper makes an initial attempt in establishing a solution theory in the spectral Barron space for a class of elliptic PDEs. 

The analysis of the generalization error is also intimately related to the function class (e.g. neural networks) we use, in particular its complexity. This makes the generalization analysis quite different from the analysis of numerical quadrature error in an usual finite element method. We face a trade-off between the approximation and generalization: To reduce the approximation error, one would like to use an approximation ansatz which involves large number of degrees of freedom, however, such choice will incur large generalization error. 

The training of the neural networks also remains to be a very challenging problem since the associated optimization problem is highly non-convex. In fact, even under a standard supervised learning setting, we still largely lack understanding of the optimization error, except in simplified setting where the optimization dynamics is essentially linear (see e.g., \cite{jacot2018neural,chizat2019lazy,ghorbani2019limitations}). The analysis for PDE problems would face similar, if not severer, difficulties, and it is beyond the scope of our current work. 

\smallskip

In this work, we provide a rigorous analysis to the approximation and generalization errors of the DRM for high dimensional elliptic PDEs. 
We will focus on relative simple PDEs (Poisson equation and static Schr\"odinger equation) to better convey the idea and illustrate the framework, without bogging the readers down with technical details. Our analysis, as already suggested by the discussions above, which is based on identifying a correct functional analysis setup and developing the corresponding \emph{a priori} analysis and complexity estimates, will provides dimension-independent generalization error estimates. 

\subsection{Related Works}

Several previous works on analysis of neural-network based methods for high-dimensional PDEs focus on the aspect of representation, i.e., whether a solution to the PDE can be approximated by a neural network with quantitative error control; see e.g., \cite{grohs2018proof, hutzenthaler2020proof}. 
Fixing an approximation space, the generalization error can be controlled by analyzing complexity such as covering numbers, see e.g., \cite{berner2020analysis} for a specific PDE problem.   

More recently, several papers \cite{shin2020convergence, mishra2020estimates, shin2020error,luo2020two} considered the generalization error analysis of the physics informed neural network (PINNs) approach  based on residual minimization for solving PDEs \cite{lagaris1998artificial,raissi2019physics}. In particular, the work 
\cite{shin2020convergence} established the consistency of the loss function such that the approximation converges to the true solution as the training sample increases under the assumption of vanishing training error. 
For the generalization error, Mishra and Molinaro \cite{mishra2020estimates} carried out an a-posteriori-type generalization error analysis for PINNs, and proved that the generalization error is bounded by the training error and quadrature error under some stability assumptions of the PDEs. To avoid the issue of curse of dimensionality in quadrature error, the authors also considered the cumulative generalization error which involves a validation set. The paper \cite{shin2020error} proved both a priori and a posterior estimates for residual minimization methods in Sobolev spaces. The paper \cite{luo2020two} obtained a priori generalization estimates for a class of second order linear PDEs by assuming (but without verifying) that the exact solutions of PDEs belong to a Barron-type space introduced in \cite{e2019barron}.

Different from the previous generalization error analysis, we derive a priori and dimension-explicit generalization error estimates under the assumption that the solutions of the PDEs lie in the spectral Barron space that is more aligned with  \cite{barron1993universal}. Moreover, we justify such assumption by developing a novel solution theory in the spectral Barron space for the PDEs of consideration. This regularity theory is the main difference between our work compared with the above mentioned ones. 

\smallskip 

It is worth mentioning that in a very recent preprint \cite{ewojtowytsch2020some}, E and Wojtowytsch  considered the regularity theory of high dimensional PDEs on the whole space (including screened Poisson equation, heat equation, and a viscous Hamilton-Jacobi equation) defined in the Barron space introduced by \cite{e2019barron}. Their result shared a similar spirit as our analysis of PDE regularity theory in the spectral Barron space (Theorem~\ref{thm:complexpoisson} for Poisson equation and Theorem~\ref{thm:complexschr} for static Schr\"odinger equation), while we focus on PDEs on finite domain, and as a result, we have to develop different Barron function spaces from those used for the whole space. The authors of \cite{ewojtowytsch2020some} also provided some counterexamples to regularity theory for PDE problems defined on non-convex domains, while we would only focus on simple domain (in fact hypercubes) in this work.

While we focus on the variational principle based approach for solving high dimensional PDEs using neural networks, we note that many other approaches have been developed, such as the deep BSDE method based on the control formulation of parabolic PDEs \cite{han2018solving}, 
the deep Galerkin method based on the weak formulation \cite{sirignano2018dgm}, methods based on the strong formulation (residual minimization) such as the PINNs \cite{lagaris1998artificial, raissi2019physics}, the diffusion Monte Carlo type approach for high-dimensional  eigenvalue problems \cite{Han2020solving}, just to name a few. It would be interesting future directions to extend our analysis to these methods. 

\subsection{Our Contributions} We analyze the generalization error of two-layer neural networks for solving two simple elliptic PDEs in the framework of DRM. Specifically we make the following contributions:

\begin{itemize}
    \item We define a spectral Barron space $\mB^s(\Omega)$ on a $d$-dimensional unit hypercube $\Omega = [0,1]^d$ that extend the Barron's original function space \cite{barron1993universal} from the whole space to bounded domain; see the definition \eqref{eq:barrons}. In the generalization theory we develop, we assume that the solutions lie in the spectral Barron space. 
    
    \item We show that the spectral Barron functions $\mB^2(\Omega)$ can be well approximated in the $H^1$-norm by two-layer neural networks with either ReLU or Softplus activation functions without curse of dimensionality.  Moreover, the parameters (weights and biases) of the two-layer neural networks  are controlled explicitly in terms of the spectral Barron norm. The bounds on the  neural-network  parameters play an essential role in controlling the generalization error of the neural nets. See Theorem~\ref{thm:relu} and Theorem~\ref{thm:sp} for the approximation results. 
    
    \item We derive generalization error bounds of the neural-network solutions for solving Poisson equation and the static Schr\"odinger equation under the assumption that the solutions belong to the Barron space $\mB^2(\Omega)$. We emphasize that the convergence rates in our generalization error bounds are dimension-independent and that the prefactors in the error estimates depend at most polynomially on the dimension and the Barron norms of the solutions, indicating that the DRM overcomes the curse of dimensionality when the solutions of the PDEs are spectral Barron functions. See Theorem~\ref{thm:main1} and Theorem~\ref{thm:main2} for the generalization results. 
    
    \item Last but not the least, we develop new well-posedness theory for the solutions of Poisson and   static Schr\"odinger equations in the spectral Barron space, providing sufficient conditions to verify the earlier  assumption on the solutions made in the generalization analysis. The new solution theory can be viewed as an analog of the classical PDE theory in Sobolev or H\"older spaces. See Theorem~\ref{thm:complexpoisson} and Theorem~\ref{thm:complexschr} for the new solution theory in spectral Barron space.
\end{itemize}

\subsection{Notation} We use $|x|_p$ to denote the $p$-norm of a vector $x\in \R^d$. When $p=2$ we write $|x| = |x|_2$.  

\section{Set-Up and Main Results}
\subsection{Set-Up of PDEs}
Let $\Omega  = [0,1]^d$ be the unit hypercube on $\R^d$. Let $\partial \Omega$ be the boundary of $\Omega$. We consider the following two prototype elliptic PDEs on $\Omega$ equipped with the Neumann boundary condition: Poisson equation 
\begin{equation}
    \begin{aligned}\label{eq:Poisson}
            -\Delta u  & = f \text{ on } \Omega,\\ 
     \frac{\partial u}{\partial \nu}  & =  0 \text{ on } \partial \Omega
    \end{aligned}
\end{equation}
and  the static Schr\"odinger equation
\begin{equation}\label{eq:schrneumann}
\begin{aligned}
   -\Delta u + V u & = f \text{ on } \Omega,\\ 
     \frac{\partial u}{\partial \nu} & =  0 \text{ on } \partial \Omega. 
\end{aligned}
\end{equation}
Throughout the paper, we make the minimal assumption that $f\in L^2(\Omega)$ and  $V\in L^\infty(\Omega)$ with $V(x) \geq V_{\min}>0$, although later we will impose stronger regularity assumptions on $f$ and $V$. In particular, in our high dimensional setting, we would certainly need to restrict the class of $f$ and $V$, otherwise just prescribing such general functions numerically would already incur curse of dimensionality. 
The well-posedness of the solutions to the Poisson equation and static Schr\"odinger equation in the Sobolev space $H^1(\Omega)$ as well as the variational characterizations of the solutions are well-known and 
are summarized in the proposition below, whose proof can be found in Appendix \ref{sec:proppde}. \begin{proposition}\label{prop:pde}
 (i) Assume that $f\in L^2(\Omega)$ with $\int_{\Omega} f dx = 0$. Then there exists a unique weak solution $u^\ast_P\in H^1_{\diamond} (\Omega) := \{u\in H^1(\Omega) | \int_{\Omega} u dx  = 0\}$ to the Poisson equation  \eqref{eq:Poisson}. Moreover, we have that 
 \begin{equation}\label{eq:poissonneumann2}
    u^\ast_P =  \argmin_{u\in H^1(\Omega)} \mE_P(u) :=  \argmin_{u\in H^1(\Omega)}  \Big\{\frac{1}{2} \int_{\Omega} |\nabla u|^2   dx + \frac{1}{2}\Big(\int_{\Omega} u dx\Big)^2 - \int_{\Omega} f u dx \Big\},
\end{equation}
and that for any $u\in H^1(\Omega)$,
\begin{equation}\label{eq:equiv1}
    2(\mE(u)-\mE(u^\ast_P))\leq \|u-u^\ast_P\|^2_{H^1(\Omega)} \leq  2 \max\{2C_P +1, 2\} (\mE(u)-\mE(u^\ast_P)),
\end{equation}
where $C_P$ is the Poincar\'e constant on the domain $\Omega$, i.e.,
for any $v\in H^1(\Omega)$,
$$
\Big\|v- \int_{\Omega} v dx  \Big\|_{L^2(\Omega)}^2 \leq C_P \|\nabla v\|_{L^2(\Omega)}^2 .
$$
   
  (ii) Assume that $f, V\in L^\infty(\Omega)$ and that $0< V_{\min} \leq V(x) \leq V_{\max}<\infty$ for some constants $V_{\min}$ and $ V_{\max}$. Then there exists a unique weak solution $u^\ast_S\in H^1 (\Omega) $ to the  static Schr\"odinger equation \eqref{eq:schrneumann}. Moreover, we have that
   \begin{equation}\label{eq:schrneumann2}
    u^\ast_S =  \argmin_{u\in H^1(\Omega)} \mE_S(u) :=  \argmin_{u\in H^1(\Omega)}  \Big\{\frac{1}{2} \int_{\Omega} |\nabla u|^2  +   V |u|^2 \ dx- \int_{\Omega} f u dx \Big\},
\end{equation}
and that for any $u\in H^1(\Omega)$
\begin{equation}\label{eq:equiv2}
\frac{2}{\max(1,V_{\max})} (\mE(u)-\mE(u^\ast_S))\leq \|u-u^\ast\|^2_{H^1(\Omega)} \leq  \frac{2}{\min(1,V_{\min})} (\mE(u)-\mE(u^\ast_S)).
\end{equation}
\end{proposition}

The variational formulations \eqref{eq:poissonneumann2} and \eqref{eq:schrneumann2} are the basis of the DRM \cite{weinan2018deep} for solving those PDEs. The main idea is to train neural networks to minimize the (population) loss defined by the Ritz energy functional $\mE$. 
More specifically, let $\mF\subset H^1(\Omega)$ be a hypothesis function class  parameterized by neural networks. The DRM seeks the optimal solution to the population loss $\mE$ within the hypothesis space $\mF$.   
However, the population loss requires evaluations of $d$-dimensional integrals, which can be prohibitively expensive when $d\gg 1$ if traditional quadrature methods were used. To circumvent the curse of dimensionality, it is natural to employ the Monte-Carlo method for computing the high dimensional integrals, which  leads to the so-called {\em empirical loss (or risk) minimization.}

\subsection{Empirical Loss Minimization}
Let us denote by $\mP_\Omega$  the uniform probability distributions on the domain $\Omega$. Then the loss functional $\mE_P$ and $\mE_S$ can be rewritten in terms of expectations under $\mP_\Omega$ as  
$$
\begin{aligned}
  \mE_P (u) & = |\Omega|\cdot  \bE_{X\sim \mP_\Omega} \Big[ \frac{1}{2} |\nabla u(X)|^2 - f(X) u(X)\Big] + \frac{1}{2}  \Big(|\Omega|\cdot \bE_{X\sim \mP_\Omega} u(X) \Big)^2,\\
  \mE_S (u) & = |\Omega|\cdot  \bE_{X\sim \mP_\Omega} \Big[ \frac{1}{2} |\nabla u(X)|^2 + \frac{1}{2} V(X)|u(X)|^2 - f(X) u(X)\Big] .
\end{aligned}
$$
To define the empirical loss, let $\{X_j\}_{j=1}^n$ be an i.i.d.~sequence of random variables distributed according to $\mP_\Omega$. Define the empirical losses $\mE_{n,P}$ and $\mE_{n,S}$ by setting 
\begin{equation}\begin{aligned}
    \mE_{n,P}(u)  & = \frac{1}{n} \sum_{j=1}^n\Big[ |\Omega| \cdot  \Big(\frac{1}{2} |\nabla u(X_j)|^2 - f(X_j)u(X_j) \Big)\Big] + \frac{1}{2} \Big( \frac{|\Omega|}{n}  \sum_{j=1}^n u(X_j) \Big)^2,\\
      \mE_{n,S}(u)  & = \frac{1}{n} \sum_{j=1}^n\Big[ |\Omega| \cdot  \Big(\frac{1}{2} |\nabla u(X_j)|^2 + \frac{1}{2} V(X_j) |u(X_j)|^2- f(X_j)u(X_j) \Big)\Big].
    \end{aligned}
\end{equation}
Given an empirical loss $\mE_n$, the empirical loss minimization algorithm seeks $u_{n}$ which minimizes $\mE_{n}$, i.e.
\begin{equation}\label{eq:uerm}
    u_{n} = \argmin_{u\in \mF} \mE_{n}(u).
\end{equation}
Here we have suppressed the dependence of $u_n$ on $\mF$. We denote by $u_{n,P}$ and $u_{n,S}$ the minimal solutions to the empirical loss $\mE_{n,P}$ and $\mE_{n,S}$, respectively. 

\subsection{Main Results} The goal of the present paper is to obtain quantitative  estimates for the generalization error  between 
 the minimal solution $u_{n,S}$ and $u_{n,P}$ computed from the finite data points $\{X_j\}_{j=1}^n$ and the exact solutions when the spacial dimension $d$ is large.  Our primary interest is to derive such estimates which scales mildly with respect to the increasing dimension $d$. To this end, it is necessary to assume that the true solutions lie in a smaller space which has a lower complexity than Sobolev spaces. Specifically we will  consider the spectral Barron space defined below via the cosine transformation. 
 
 Let $\mathscr{C}$ be a set of cosine functions defined by  
 \begin{equation}\label{eq:cosbasis1}
\begin{aligned}
   \mathscr{C} & := \Bigl\{ \Phi_k \Bigr\}_{k\in \N_0^d}:=  \Bigl\{ \prod_{i=1}^d \cos(\pi k_i x_i) \ |\ k_i \in \N_0 \Bigr\}.
   \end{aligned}
\end{equation}
Given $u\in L^1(\Omega)$, let $\{\hat{u}(k)\}_{k\in \N_0^d}$ be the expansion coefficients of $u$ under the basis $\{\Phi_k\}_{k\in \N_0^d}$.
Let us define for $s\geq 0$ the spectral Barron space $\mB^s(\Omega)$ on $\Omega$ by 
\begin{equation}\label{eq:barrons}
    \mB^s(\Omega) := \Big\{u\in L^1(\Omega): \sum_{k\in \N^d_0} (1 + \pi^{s}|k|_1^{s}) |\hat{u}(k)| < \infty \Big\}.
\end{equation}
The spectral Barron norm of a function $u$ on $\mB^s(\Omega)$ is given by  
$$
\|u\|_{\mB^s(\Omega)}  = \sum_{k\in \N^d_0} (1 + \pi^{s}|k|_1^{s}) |\hat{u}(k)|.
$$
Observe that a function $f\in \mB^s(\Omega)$ if and only if $\{\hat{u}(k)\}_{k\in \N_0^d}$ belongs to the weighted $\ell^1$-space $\ell^1_{W_s}(\N_0^d)$ on the lattice $\N_0^d$ with the weights $W_s(k) = (1 + \pi^{2s}|k|_1^{2s})$. When $s=2$, we adopt the short notation $\mB(\Omega)$ for $\mB^2(\Omega)$. Our definition of spectral Barron space is strongly motivated by the seminar work by Barron \cite{barron1993universal} and other  recent works \cite{klusowski2018approximation,e2019barron,siegel2020approximation}. The initial Barron function $f$ in \cite{barron1993universal} is defined on the whole space $\R^d$ whose Fourier transform $\hat{f}(w)$ satisfies that $\int |\hat{f}(\omega)| |\omega| d\omega < \infty$. Our spectral Barron space $\mB^s(\Omega)$ with $s=1$ can be viewed as a discrete analog of the initial  Barron space from  \cite{barron1993universal}. 

The most important property of the Barron functions is that those functions can be well approximated by two-layer neural networks without  the curse of dimensionality. To make this more precise, let us define  the class of two-layer neural networks to be used as our hypothesis space for solving PDEs. Given an activation function $\phi$, a constant $B>0$ and the number of hidden neurons $m$, we define
 \begin{equation}\label{eq:fphim}
\mF_{\phi, m}(B) := \Big\{c + \sum_{i=1}^m \gamma_i \phi(\omega_i \cdot x  - t_i), |c|\leq 2B, |w_i|_1=1, |t_i|\leq 1, \sum_{i=1}^m |\gamma_i|\leq 4B\Big\}.
\end{equation}
Our first result concerns the approximation of spectral Barron functions in $\mB(\Omega)$ by two-layer neural networks with $\relu$ activation functions. 

\begin{theorem}\label{thm:relu}
Consider the class of two-layer ReLU neural networks  
\begin{equation}\label{eq:frelum}
\mF_{\relu,m}(B) := \Big\{c + \sum_{i=1}^m \gamma_i \relu(\omega_i \cdot x  - t_i), |c|\leq 2B, |w_i|_1=1, |t_i|\leq 1, \sum_{i=1}^m |\gamma_i|\leq 4B\Big\}.
\end{equation}
Then for any $u\in \mB(\Omega)$, there exists 
$u_m\in \mF_{\relu,m}(\|u\|_{\mB(\Omega)}) $, 
such that 
$$
\|u - u_m\|_{H^1(\Omega)} \leq \frac{\sqrt{116} \|u\|_{\mB(\Omega)}}{\sqrt{m}}.
$$
\end{theorem}

A similar approximation result was firstly proved in  the seminar paper of Barron \cite{barron1993universal} where the same approximation rate $O(m^{-\frac{1}{2}})$ was also obtained when approximating  the Barron function defined on the whole space with two-layer neural nets with the sigmoid activation function in the $L^\infty$-norm.  Results of this kind were also obtained in the recent works  \cite{klusowski2018approximation,e2019barron,siegel2020approximation}. In particular, the same convergence rate was proved for approximating  functions $f$ with $\|f\|_{\mathscr{B}^s} = \int_{\R^d} |\hat{f}(\omega)| (1 + |\omega|)^s d\omega   < \infty$ in  Sobolev norms by two-layer networks with a general class of  activation functions satisfying polynomial decay condition. The convergence rate $O(m^{-\frac{1}{2}})$ was recently  improved to $O(m^{-(\frac{1}{2} + \delta(d))})$ with $\delta(d)>0$ depending on $d$ in \cite{siegel2020high} when $\text{ReLU}^k$ or cosine is used as the activation function. Moreover, the rate has been proved to be sharp in the Sobolev norms when the index $s$ of Barron space and that of the Sobolev norm belong to certain appropriate regime. 

 Although the function class $\mF_{\relu,m}(B)$ can be used to approximate functions in $\mB(\Omega)$ without curse of dimensionality, it brings several issues to both theory and computation if used as the hypothesis class for solving PDEs. On the one hand, the set $\mF_{\relu,m} \subset H^1(\Omega)$ consists of only piecewise affine functions, which may be undesirable in some PDE problems if the function of interest is expected to be more regular or smooth. On the other hand, the fact that $\mF_{\relu,m}$ only admits first order weak derivatives makes it extremely difficult to bound the complexities of function classes involving derivatives of functions from  $\mF_{\relu,m}$, whereas the latter is a crucial ingredient for getting a generalization bound for the DRM. 
 
 To resolve those issues, in what follows we will consider instead a class of two-layer neural networks with the Softplus \cite{dugas2001incorporating,glorot2011deep} activation function. Recall the Softplus function
$
\sp(z) = \ln(1+ e^{z})
$
and its rescaled version $\sp_{\tau}(z) $ defined also for $\tau>0$, 
$$
\sp_{\tau}(z) = \frac{1}{\tau} \sp(\tau z) = \frac{1}{\tau}\ln(1 + e^{\tau z}).
$$
Observe that the rescaled Softplus $\sp_{\tau}(z)$  can be viewed as a smooth approximation of the ReLU function since $\sp_{\tau}(z) \gt \relu(z)$ as $\tau \gt 0$ for any $z\in \R$ (see Lemma \ref{lem:relusp} for a quantitative statement). Moreover, the two-layer neural networks with the activation function $\sp_{\tau}$ satisfy a similar approximation result as Theorem \ref{thm:relu} when approximating spectral Barron functions in $\mB(\Omega)$, as shown in the next theorem. 

\begin{theorem}\label{thm:sp}
Consider the class of two-layer Softplus neural networks functions 
\begin{equation}\label{eq:fspm}
\mF_{\sp_{\tau},m}(B) := \Big\{c + \sum_{i=1}^m \gamma_i \sp_{\tau}(\omega_i \cdot x  - t_i), |c|\leq 2B, |w_i|_1=1, |t_i|\leq 1, \sum_{i=1}^m |\gamma_i|\leq 4B\Big\}.
\end{equation}
 Then for any $u\in \mB(\Omega)$, there exists a two-layer neural network $u_m \in \mF_{\sp_{\tau},m}(\|u\|_{\mB(\Omega)})$ with activation function $\sp_{\tau}$ with $\tau  = \sqrt{m}$,
such that 
$$
\|u - u_m\|_{H^1(\Omega)} \leq  \frac{\|u\|_{\mB(\Omega)}(6\log m+30)}{\sqrt{m}}.
$$
\end{theorem}
The proofs of Theorem \ref{thm:relu} and Theorem \ref{thm:sp} can be found in Section \ref{sec:approx}. 

\medskip 

Now we are ready to state the main generalization results of two-layer neural networks for solving Poisson and the static Schr\"odinger equations.  We start with the generalization error bound for the neural-network solution in the Poisson case. 
\begin{theorem}\label{thm:main1}
  Assume that the solution $u^\ast_P$ of the Neumann problem for the Poisson equation \eqref{eq:Poisson} satisfies that $\|u^\ast_P\|_{\mB(\Omega)} <\infty$. Let $u_{n,S}^m$ be the minimizer of  the empirical loss $\mE_{n,P}$ in the set $\mF  = \mF_{\sp_\tau, m}(\|u^\ast_P\|_{\mB(\Omega)})$  with $\tau = \sqrt{m}$. Then it holds that 
  \begin{equation}\label{eq:mainerr1}
      \bE \big[\mE_P(u_{n,P}^m) - \mE_P(u^\ast_P)\big] \leq  \frac{C_1 \sqrt{m} (\sqrt{\log m}+1)}{\sqrt{n}} + \frac{C_2(\log m+1)^2}{m}.
  \end{equation}
  Here $C_1>0$ depends polynomially on $ \|u^\ast_P\|_{\mB(\Omega)}, d,\|f\|_{L^\infty(\Omega)},$ and $C_2>0$ depends quadratically on $\|u^\ast_P\|_{\mB(\Omega)}$. In particular, setting $m = n^{\frac{1}{3}}$ in \eqref{eq:mainerr1}  leads to 
  $$
  \bE \big[\mE_P(u_{n,P}^m) - \mE_P(u^\ast)\big] \leq \frac{C_3(\log n)^2}{n^{\frac{1}{3}}}
  $$
  for some $C_3>0$ depending only polynomially on $ \|u^\ast_P\|_{\mB(\Omega)}, d,\|f\|_{L^\infty(\Omega)}$.
\end{theorem}

Next we state the generalization error for the neural-network solution in the case of the  static Schr\"odinger equation. 

\begin{theorem}\label{thm:main2}
  Assume that the solution $u^\ast_S$ of the Neumann problem for the static Schr\"odinger equation \eqref{eq:schrneumann} satisfies that $\|u^\ast_S\|_{\mB(\Omega)} <\infty$. Let $u_{n,S}^m$ be the minimizer of  the empirical loss $\mE_{n,S}$ in the set  $\mF  = \mF_{\sp_\tau, m}(\|u^\ast_S\|_{\mB(\Omega)})$  with $\tau = \sqrt{m}$. Then it holds that 
  \begin{equation}\label{eq:mainerr2}
      \bE \big[\mE_S(u_{n,S}^m) - \mE_S(u^\ast_S)\big] \leq  \frac{C_4 \sqrt{m} (\sqrt{\log m}+1)}{\sqrt{n}} + \frac{C_5(\log m+1)^2}{m}.
  \end{equation}
  Here  $C_4>0$ depends polynomially on $ \|u^\ast_S\|_{\mB(\Omega)}, d,\|f\|_{L^\infty(\Omega)}, \|V\|_{L^\infty(\Omega)}$ and $C_5>$ depends quadratically on $\|u^\ast_S\|_{\mB(\Omega)}$. In particular, setting $m = n^{\frac{1}{3}}$ in \eqref{eq:mainerr2}  leads to 
  $$
  \bE \big[\mE_S(u_{n,S}^m) - \mE_S(u^\ast_S)\big] \leq \frac{C_6(\log n)^2}{n^{\frac{1}{3}}}
  $$
  for some $C_6>0$ depending only polynomially  on $ \|u^\ast_S\|_{\mB(\Omega)}, d,\|f\|_{L^\infty(\Omega)},\|V\|_{L^\infty(\Omega)}$.
\end{theorem}

\begin{remark}
  Thanks to the estimates \eqref{eq:equiv1} and \eqref{eq:equiv2}, the generalization bound above on the  energy excess translate directly to the generalization bound on square of the  $H^1$-error between the neural-network solution and the exact solution of the PDE. Specifically, when $m = n^{\frac{1}{3}}$, it holds that for some constant $C_7>0$,
  $$
  \bE \|u_{n}^m-u^\ast\|^2_{H^1(\Omega)} \leq C_7 \frac{(\log n)^2}{n^{\frac{1}{3}}}.
  $$
\end{remark}

Theorem \ref{thm:main1} and Theorem \ref{thm:main2} show that the generalization error of the neural-network solution for Poisson and the static Schr\"odinger equations do not suffer from the curse of dimensionality under the key assumption that their exact solutions  belong to the spectral Barron space $\mB(\Omega)$.
The proofs of Theorem \ref{thm:main1} and Theorem \ref{thm:main2} can be found in Section \ref{sec:proofmain}.

Finally we verify the key low-complexity assumption by proving new well-posedness theory of  Poisson and the static Schr\"odinger equations in  spectral Barron spaces. 
We start with the new solution theory for the Poisson equation, whose proof  can be found in Section~\ref{sec:complexpoisson}. 
\begin{theorem}\label{thm:complexpoisson}
Assume that $f\in \mB^s(\Omega)$ with $s\geq 0$  and that $\hat{f}_0 = \int_{\Omega} f(x)dx = 0$. Then the unique solution $u^\ast$ to the Neumann problem for the Poisson equation satisfies that $u^\ast \in \mB^{s+2}(\Omega)$ and that 
$$
\|u^\ast\|_{\mB^{s+2}(\Omega)} \leq d \|f\|_{\mB^s(\Omega)}.
$$
In particular, when $s=0$ we have $\|u^\ast\|_{\mB(\Omega)} \leq d \|f\|_{\mB^0(\Omega)}$.
\end{theorem}
The next theorem establishes the solution theory  for the static Schr\"odinger equation in spectral Barron spaces. 
\begin{theorem}\label{thm:complexschr}
Assume that $f\in \mB^s(\Omega)$ with $s\geq 0$  and that $V\in \mB^s(\Omega)$ with $V(x) \geq V_{\min} > 0$ for every $x\in \R^d$. Then the static Schr\"odinger problem \eqref{eq:schrneumann} has a unique solution $u\in \mB^{s+2}(\Omega)$. Moreover, there exists a  constant $C_{8}>0$ depending on $V$ and $d$ such that 
\begin{equation}\label{eq:complexschr}
    \|u\|_{\mB^{s+2}(\Omega)} \leq C_{8} \|f\|_{\mB^s(\Omega)}.
\end{equation}
In particular, when $s=0$ we have $\|u^\ast\|_{\mB(\Omega)} \leq C_{8} \|f\|_{\mB^0(\Omega)}$.
\end{theorem}
The stability estimates above can be viewed as an analog of the standard Sobolev regularity estimate $\|u\|_{H^{s+2}(\Omega)} \leq C \|f\|_{H^s(\Omega)}$. However, the proof of the estimate \eqref{eq:complexschr} is quite different from that of the Sobolev estimate. In particular, due to the lack of Hilbert structure in the Barron space $\mB^s(\Omega)$, the standard Lax-Milgram theorem and the bootstrap arguments for proving the Sobolev regularity estimates can not be applied here. Instead, we turn to studying the equivalent operator equation satisfied by the cosine coefficients of the solution of the static Schr\"odinger equation. By exploiting the fact that the Barron space is a weighted $\ell^1$-space on the cosine coefficients, we manage to prove the well-posedness of the operator equation and the stability estimate \eqref{eq:complexschr} with an application of the Fredholm theory to the operator equation. The complete proof of Theorem \ref{thm:complexschr} can be found in Section \ref{sec:proofcomplexS}.

\subsection{Discussions and Future Directions}
We established dimension-independent rates of convergence for the generalization error of the DRM
for solving two simple linear elliptic PDEs. We would like to discuss some restrictions of the main results and point out some interesting future directions. 

First, some numerical results show that the convergence rates in our generalization error estimates may not be sharp. In fact, Siegel and Xu  \cite{siegel2020high} obtained sharp convergence rates of $O(m^{-(\frac{1}{2} + \delta(d))})$ with some $\delta(d)>0$ for approximating a similar class of spectral Barron functions using two-layer neural nets with cosine and $\relu^k$ activation functions. However, the parameters (weights and biases) of the neural networks  constructed in their approximation results were not well controlled (and maybe unbounded) and potentially could lead to large generalization errors. One interesting open question is to sharpen the approximation rate for our spectral Barron functions using controllable two-layer neural networks with possibly different activation functions. On the other hand, the  statistical error bound $O(\frac{\sqrt{m}(\sqrt{\log m} +1)}{\sqrt{n}})$ may also be improved with  sharper and more delicate Rademacher complexity estimates of the neural networks. 

We restricted our attention on two simple elliptic  problems defined on a hypercube with the Neumann boundary condition to better convey the main ideas. It is natural to consider carrying out similar programs of solving more general PDE problems defined on general bounded or unbounded domains with other boundary conditions. One major  difficulty arises when one comes to the definition of Barron functions on a general bounded domain and  our spectral Barron functions built on cosine expansions can not be adapted to general domains. 
Other Barron functions such as the one defined 
in \cite{e2019barron} via  integral representation are  on bounded domains and may be considered as alternatives, but building a solution theory for PDEs in those spaces seems highly nontrivial; see \cite{ewojtowytsch2020some} for some results and discussions along this direction. Another major issue comes from solving PDEs with essential boundary conditions such as Dirichlet or periodic boundary conditions, where one needs to construct neural networks that satisfy those boundary conditions; we refer to \cite{ozbay2019poisson,dong2020method} for some initial attempts in this direction.

Finally,  the analysis of training error of neural network methods for solving PDEs is a highly important and challenging question.  The difficulty is largely due to the non-convexity of the loss function in the parameters. Nevertheless, recent breakthroughs in the theoretical analysis of two-layer neural networks training show that the training dynamics can be largely simplified in infinite-width limit, such as in the the mean field regime \cite{rotskoff2018parameters,mei2018mean,sirignano2020mean,chizat2018global} or neural tangent kernel (NTK) regime \cite{jacot2018neural, chizat2019lazy, ghorbani2019limitations}, where global convergence of limiting dynamics can be proved under suitable assumptions. It is an exciting direction to establish similar convergence results for overparameterized two-layer networks in the context of solving PDEs.

\section{Abstract generalization error bounds} In this section, we derive some abstract generalization bounds for the  empirical loss minimization discussed in the previous section. To simply the notation, we suppress the problem-dependent subscript  $P$ or $S$ and denote by $u_n$ the minimizer of the empirical loss $\mE_n$ over the hypothesis space $\mF$. Recall  that $u^\ast$ is the exact solution of the PDE. We aim to bound the energy excess  
$$
\Delta \mE_n := \mE (u_n) - \mE(u^\ast).
$$
By definition we have that $\Delta \mE_n \geq 0$. To bound $\Delta \mE_n$ from above, we first decompose $\Delta \mE_n$ as 
\begin{equation}\label{eq:dE}
    \Delta \mE_n = 
    \mE (u_n) 
    - \mE_n(u_n) + \mE_n(u_n) - \mE_n(u_\mF) +
    \mE_n(u_\mF) - 
    \mE(u_\mF) +
    \mE(u_\mF)
    - \mE(u^\ast).
\end{equation}
Here $u_{\mF} = \argmin_{u\in \mF} \mE (u)$. 
Since $u_n$ is the minimizer of $\mE_n$, $\mE_n(u_n) - \mE_n(u_\mF) \leq 0$. Therefore taking expectation
on both sides of  \eqref{eq:dE} gives 
\begin{equation}\label{eq:expdeltaen}
\bE \Delta \mE_n \leq \underbrace{\bE [\mE (u_n) 
    - \mE_n(u_n)]}_{\Delta \mE_{\text{gen}}}
    + \underbrace{\bE [\mE_n(u_\mF)] - \mE(u_\mF)}_{\Delta \mE_{\text{bias}}}
    + \underbrace{\mE(u_\mF)
    - \mE(u^\ast)}_{\Delta \mE_{\text{approx}}}.
\end{equation}
Observe that $\Delta \mE_{\text{gen}}$ and $ \Delta \mE_{\text{bias}}$ are the statistical errors: the first term $\Delta \mE_{\text{gen}}$ describing the generalization error of the  empirical loss minimization over the hypothesis space $\mF$ and the second term $\Delta \mE_{\text{bias}}$ being the bias coming from the Monte Carlo approximation of the integrals. Whereas the third term $\Delta \mE_{\text{approx}}$ is the  approximation error incurred by restricting minimizing $\mE$ from over the set $H^1(\Omega)$ to $\mF$. Moreover, thanks to  Proposition \ref{prop:pde}, the third term $\Delta \mE_{\text{approx}}$ is equivalent  (up to a constant) to $\inf_{u\in \mF}\|u - u^\ast\|^2_{H^1(\Omega)}$.

To control the statistical errors, it is essential to prove the so-called uniform law of large numbers for certain function classes, where the notion of {\em Rademacher complexity} plays an important role, which we now recall below.

\begin{definition}
We define for a set of random variables $\{Z_j\}_{j=1}^n$ independently distributed according to $\mP_{\Omega}$ and a function class $\mS$  the random variable
$$
\hat{R}_n (\mS) := \bE_{\sigma} \Bigl[\sup_{g\in \mS} \Big| \frac{1}{n}\sum_{j=1}^n \sigma_j g(Z_j) \Big| \;\Big|\; Z_1, \cdots, Z_n\Bigr],
$$
where the expectation $\bE_\sigma$ is taken with respect to the independent uniform Bernoulli sequence $\{\sigma_j\}_{j=1}^n$ with $\sigma_j \in \{\pm 1\}$. Then the Rademacher complexity of $\mS$ defined by $R_n (\mS) = \bE_{\mP_\Omega} [\hat{R}_n (\mS) ]$. 
\end{definition}

The following important symmetrization lemma makes the connection between the uniform law of large numbers and the Rademacher complexity.

\begin{lemma}\label{lem:radcomp}\cite[Proposition 4.11]{wainwright2019high}
    Let $\mF$ be a set of functions. Then  $$
     \bE \sup_{u\in \mF} \Big|\frac{1}{n}  \sum_{j=1}^n u(X_j) - \bE_{X\sim \mP_\Omega} u(X)  \Big|
     \leq 2  R_n(\mF).
    $$
\end{lemma}

\subsection{Poisson Equation} In this subsection we derive the abstract generalization bound in the setting of Poisson equation.  Recall the Ritz loss and the empirical loss  associated to the Poisson equation
$$
\begin{aligned}
        \mE (u) & = |\Omega|\cdot  \bE_{X\sim \mP_\Omega} \Big[ \frac{1}{2} |\nabla u(X)|^2 - f(X) u(X)\Big] + \frac{1}{2}  \Big(|\Omega|\cdot \bE_{X\sim \mP_\Omega} u(X) \Big)^2\\
        & =: \mE^1(u) + \mE^2(u),\\
        \mE_{n}(u)  & = \frac{1}{n} \sum_{j=1}^n\Big[ |\Omega| \cdot  \Big(\frac{1}{2} |\nabla u(X_j)|^2 - f(X_j)u(X_j) \Big)\Big] + \frac{1}{2} \Big( \frac{|\Omega|}{n}  \sum_{j=1}^n u(X_j) \Big)^2\\
         & =: \mE^1_n(u) + \mE^2_n(u).
\end{aligned}
$$
By definition, the bias term $\Delta \mE_{\text{bias}}$ satisfies that 
$$
\begin{aligned}
  \Delta \mE_{\text{bias}} & = \bE [\mE_n^1(u_\mF)] - \mE^1(u_\mF) + \bE [\mE_n^2(u_\mF)] - \mE^2(u_\mF)\\ 
  & = \frac{1}{2} \bE \Big( \frac{|\Omega|}{n}  \sum_{j=1}^n u(X_j) \Big)^2  - \frac{1}{2}  \Big(|\Omega|\cdot \bE_{X\sim \mP_\Omega} u(X) \Big)^2 \\ 
  & = \frac{1}{2} \bE \Big[ \Big(\frac{1}{n}  \sum_{j=1}^n u_{\mF}(X_j) - \bE_{X\sim \mP_\Omega} u_{\mF}(X)  \Big) \cdot
  \Big(\frac{1}{n}  \sum_{j=1}^n u_{\mF}(X_j) + \bE_{X\sim \mP_\Omega} u_{\mF}(X)  \Big) \Big]\\
  & \leq  \|u_\mF\|_{L^\infty(\Omega)} \cdot \bE \sup_{u\in \mF} \Big|\frac{1}{n}  \sum_{j=1}^n u(X_j) - \bE_{X\sim \mP_\Omega} u(X)  \Big|\\
  & \leq 2 \sup_{u\in\mF}\|u\|_{L^\infty(\Omega)} \cdot R_n(\mF),
 \end{aligned}
$$
where we have used $\lvert\Omega\rvert = 1$ the last inequality follows from Lemma \ref{lem:radcomp}.

Next we bound the first term  $\Delta \mE_{\text{gen}}$.
Let us first define the set of functions $\mG_P$ for the term appeared in $\mE^1$ by
$$
\begin{aligned}
\mG_P & := \Big\{g: \Omega \gt \R \ \big|\  g = \frac{1}{2}  | \nabla u|^2 - f u  \text{ where }u \in \mF\Big\}.
\end{aligned}
$$
Then it follows by Lemma \ref{lem:radcomp} that 
$$
\begin{aligned}
\Delta \mE_{\text{gen}} & \leq \bE \sup_{v\in \mF} \Big|\mE(v)  - \mE_n(v)\Big| \\
& \leq \bE \sup_{v\in \mF} \Big|\mE^1(v)  - \mE^1_n(v)\Big| +  
\bE \sup_{v\in \mF} \Big|\mE^2(v)  - \mE^2_n(v)\Big|
\\
&\leq \bE \sup_{g\in \mG} \Big| \frac{1}{n}\sum_{j=1}^n g(X_j) - \bE_{\mP_{\Omega}}[g]  \Big| +  
\bE \sup_{u\in \mF} \frac{1}{2} \Big|
\Big(\bE_{X\sim \mP_\Omega} u(X)\Big)^2 - 
\Big(\frac{1}{n}\sum_{j=1}^n u(X_j) \Big)^2
\Big|
\\
& \leq 2 R_n(\mG_P) + \sup_{u\in \mF}\|u\|_{L^\infty(\Omega)}  \cdot  \bE \sup_{u\in\mF} \Big|\frac{1}{n}\sum_{j=1}^n u(X_j) - \bE_{X\sim \mP_\Omega} u(X)\Big|\\
& \leq 2 R_n(\mG_P) + 2\sup_{u\in \mF}\|u\|_{L^\infty(\Omega)} R_n(\mF).
\end{aligned}
$$
Finally owing to the estimate \eqref{eq:equiv1} in  Proposition \ref{prop:pde},  the approximation error $\Delta \mE_{\text{approx}}$ satisfies that 
$$\begin{aligned}
\Delta \mE_{\text{approx}} 
& \leq \frac{1}{2} \inf_{u\in \mF}\|u - u^\ast\|^2_{H^1(\Omega)}.
\end{aligned}
$$
To summarize, we have established the following abstract generalization error bound for the energy excess  $\Delta \mE_n$ in the case of Poisson equation. 
\begin{theorem}\label{thm:ermbdP}
  Let $u_{n,P}$ be the minimizer of the empirical risk $\mE_{n,P}$ within the hypothesis class $\mF$ satisfying that $\sup_{u\in \mF} \|u\|_{L^\infty(\Omega)}<\infty$. Let $\Delta \mE_{n,P} = \mE_{P}(u_{n,P}) - \mE_{P}(u^\ast_P)$.  Then 
  \begin{equation}
      \bE \Delta \mE_{n,P} \leq 2R_n(\mG_P) + 4\sup_{u\in \mF} \|u\|_{L^\infty(\Omega)}\cdot R_n(\mF) + \frac{1}{2} \inf_{u\in \mF}\|u - u^\ast\|^2_{H^1(\Omega)}.
  \end{equation}
\end{theorem}

\subsection{Static Schr\"odinger Equation} In this subsection we proceed to prove an abstract generalization bound for the static Schr\"odinger equation. First recall the corresponding Ritz loss and the empirical loss  as follows
$$
\begin{aligned}
         \mE_S (u) & = |\Omega|\cdot  \bE_{X\sim \mP_\Omega} \Big[ \frac{1}{2} |\nabla u(X)|^2 + \frac{1}{2} V(X)|u(X)|^2 - f(X) u(X)\Big], \\
         \mE_{n,S}(u)  & = \frac{1}{n} \sum_{j=1}^n\Big[ |\Omega| \cdot  \Big(\frac{1}{2} |\nabla u(X_j)|^2 + \frac{1}{2} V(X_j) |u(X_j)|^2- f(X_j)u(X_j) \Big)\Big].
\end{aligned}
$$
Similar to the previous subsection, we introduce the function class $\mG_S$ by setting 
$$
\mG_S :=  \Big\{g: \Omega \gt \R \ \big|\  g = \frac{1}{2}  | \nabla u|^2 + \frac{1}{2} V |u|^2 - f u  \text{ where }u \in \mF\Big\}.
$$
In the Schr\"odinger  case, since the Ritz energy $\mE_S$ is linear with respect to the probability measure $\mP_\Omega$, the statistical errors $\Delta \mE_{\text{gen}}$ and $\Delta \mE_{\text{bias}}$ are simpler than those in the Poisson case. In particular, a similar calculation shows that $\Delta \mE_{\text{gen}} = 0$ and $\Delta\mE_2 \leq 2R_n(\mG_S)$.
Therefore as a result of \eqref{eq:expdeltaen} we obtained the following theorem. 
\begin{theorem}\label{thm:ermbdS}
  Let $u_{n,S}$ be the minimizer of the empirical risk $\mE_{n,S}$ within the hypothesis class $\mF$ satisfying that $\sup_{u\in \mF} \|u\|_{L^\infty(\Omega)}<\infty$. Let $\Delta \mE_{n,S} = \mE_{P}(u_{n,S}) - \mE_{P}(u^\ast_S)$.  Then 
  \begin{equation}
      \bE \Delta \mE_{n,S} \leq 2R_n(\mG_S) +\frac{1}{2} \inf_{u\in \mF}\|u - u^\ast\|^2_{H^1(\Omega)}.
  \end{equation}
\end{theorem}

\section{Spectral Barron functions on the hypercube and their $H^1$-approximation.}\label{sec:approx}
In this section, we discuss the properties of spectral Barron functions on the $d$-dimensional hypercube defined by \eqref{eq:barrons} as well as their neural network approximations. Since our spectral Barron functions are defined via the  expansion under the following set of cosine functions:
$$
\mathscr{C} = \Bigl\{ \Phi_k \Bigr\}_{k\in \N_0^d}:=  \Bigl\{ \prod_{i=1}^d \cos(\pi k_i x_i) \ |\ k_i \in \N_0 \Bigr\},
$$
we start by stating some preliminaries on $\mathscr{C}$  and the product of cosines to be used in the subsequent proofs. 
\subsection{Preliminary Lemmas}
\begin{lemma}\label{lem:basis}
The set $\mathscr{C} $  forms an orthogonal basis of $L^2(\Omega)$ and  $H^1(\Omega)$. 
\end{lemma}

\begin{proof}
First that $\mathscr{C} $ forms an orthogonal basis of $L^2(\Omega)$  follows directly from the Parseval's theorem applied to the Fourier expansion of the even extension of a function $u$ from  $L^2(\Omega)$. 
To see $\mathscr{C}$ is  an  orthogonal basis of $H^1(\Omega)$, since $\mathscr{C}$ is an orthogonal set of $H^1(\Omega)$, it suffices to show that if  $u\in H^1(\Omega)$ satisfying 
$$
\Big(u, \Phi_k \Big)_{H^1(\Omega)} = 0
$$
for all $k\in \N_0^d$, then $u = 0 $. In fact, the last display above yields that $$
\begin{aligned}
0 & = \int_{\Omega} u \cdot \Phi_k dx + \int_{\Omega} \nabla u \cdot \nabla \Phi_k dx\\
& = \int_{\Omega} u\cdot (\Phi_k - \Delta \Phi_k) dx \\
& = (1 + \pi^2|k|^2) \int_{\Omega}  u \cdot \Phi_k  dx,
\end{aligned}
$$
where for the second identity we have used the Green's formula and the fact that the normal derivative of $\Phi_k$ vanishes on the boundary of $\Omega$. Therefore we have obtained that $(u, \Phi_k)_{L^2} =0$ for any $k\in \N_0^d$, which implies that $u=0$ since $\mathscr{C}$ is an orthogonal basis of $L^2(\Omega)$.
\end{proof}

Given $u\in L^2(\Omega)$, let $\{\hat{u}(k)\}_{k\in \N_0^d}$ be the expansion coefficients of $u$ under the basis $\{\Phi_k\}_{k\in \N_0^d}$. Then for any $u\in L^2(\Omega)$, 
$$
u(x) = \sum_{k\in \N_0^d}  \hat{u}(k) \Phi_k(x).
$$
Moreover, it follows from a straightforward calculation that for $u\in H^1(\Omega)$,
$$
\|u\|_{H^1(\Omega)}^2 =  \sum_{k\in \N^d_0} \alpha_k  (1 + \pi^2 |k|^2) |\hat{u}(k)|^2 ,
$$
where $\alpha_k = \langle \Phi_k, \Phi_k\rangle_{L^2(\Omega)} =  2^{-\sum_{i=1}^d\mathbf{1}_{k_i\neq 0}} \leq 1$. 
 This implies the following characterization of a function from $H^1(\Omega)$ function in terms of its  expansion coefficients under $\mathscr{C}$.

\begin{corollary}
 The space $H^1(\Omega)$ can be characterized as
 $$
 H^1(\Omega) = \Big\{u\in L^2(\Omega)\ \Big|\  \sum_{k\in \N_0^d} |\hat{u}(k)|^2 (1 + \pi^2 |k|^2) < \infty \Big\}.
 $$
\end{corollary}

The following elementary product formula of cosine functions will also be useful.

\begin{lemma}\label{lem:cos}
For any $\{\theta_i\}_{i=1}^d \subset \R$, 
$$
\prod_{i=1}^d \cos(\theta_i) = \frac{1}{2^d}\sum_{\xi \in \Xi} \cos(\xi \cdot \theta),
$$
where $\theta = (\theta_1,\cdots, \theta_d)^T$ and $\Xi = \{1,-1\}^d$.
\end{lemma}
\begin{proof}
The lemma follows directly by iterating the following simple identity 
\begin{equation*}
    \begin{aligned}
\cos(\theta_1)\cos(\theta_2) & = \frac{1}{2} \big(\cos(\theta_1 + \theta_2) + \cos(\theta_1 - \theta_2)\big)\\
& = \frac{1}{4} \big(\cos(\theta_1 + \theta_2) + \cos(\theta_1 - \theta_2) + \cos(-\theta_1 - \theta_2) + \cos(-\theta_1 + \theta_2)\big). \qedhere
\end{aligned}
\end{equation*}
\end{proof}

\subsection{Spectral Barron Space and Neural-Network Approximation}  
Recall for any $s\in \N$ the spectral Barron space $\mB^s(\Omega)$  given by 
$$
\mB^s(\Omega) := \Bigl\{u\in L^1(\Omega): \sum_{k\in \N^d_0} (1 + \pi^{s}|k|_1^{s}) |\hat{u}(k)| < \infty \Bigr\}
$$
with associated norm $\|u\|_{\mB^s(\Omega)} := \sum_{k\in \N^d_0} (1 + \pi^{s}|k|_1^{s}) |\hat{u}(k)|$. Recall also the short notation $\mB(\Omega) $ for $\mB^2(\Omega) $.
\begin{lemma} \label{lem:embedd}
The following embedding results hold: 

(i) $\mB(\Omega) \hookrightarrow H^1(\Omega)$;

(ii) $\mB^0(\Omega) \hookrightarrow L^\infty(\Omega)$.
\end{lemma}
\begin{proof}
(i). If $u\in \mB(\Omega)$, then 
$
\|u\|_{\mB(\Omega)} = \sum_{k\in \N^d_0} (1 + \pi^2|k|_1^2) |\hat{u}(k)| < \infty. 
$
This particularly implies $|\hat{u}(k)| \leq \|u\|_{\mB(\Omega) }$ for each $k\in \N^d$. Since $\alpha_k\leq 1$, we have from the Cauchy-Schwarz inequality that 
\begin{align*}
\|u\|^2_{H^1(\Omega)}& =  \sum_{k\in \N_0^d} \alpha_k  (1 + \pi^2|k|^2) |\hat{u}(k)|^2 \\
& \leq \|u\|_{\mB(\Omega) } \sum_{k\in \N_0^d} (1 + \pi^2d |k|_1^2) |\hat{u}(k)|\\
& \leq d  \|u\|_{\mB(\Omega) }^2. 
\end{align*}

(ii). For $u\in \mB^0(\Omega)$, using the fact that $\|\Phi_k\|_{L^\infty(\Omega)}\leq 1$ we have that 
\begin{equation*}
\|u\|_{L^\infty(\Omega)} = \Big\|\sum_{k\in \N_0^d} \hat{u}(k)\Phi_k\Big\|_{L^\infty(\Omega)} \leq  
\sum_{k\in \N_0^d} |\hat{u}(k)| = \|u\|_{\mB(\Omega)}.\qedhere
\end{equation*}
\end{proof}

Thanks to Lemma \ref{lem:basis} and Lemma \ref{lem:cos}, any function $u\in H^1(\Omega)$ admits the expansion 
\begin{equation}\label{eq:expan}
u (x) = \sum_{k\in \N^d_0 }  \hat{u}(k) \cdot \frac{1}{2^d}\sum_{\xi\in \Xi}\cos(\pi k_\xi \cdot x),
\end{equation}
where $\hat{u}(k)$ is the expansion coefficient of $u$ under the basis $\mathscr{C}$ and $k_\xi = (k_1\xi_1,\cdots, k_d \xi_d)\in \Z^d$. 

Given $u \in \mB(\Omega) \subset H^1(\Omega)$, letting $(-1)^{\theta(k)} = \text{sign}(\hat{u}(k))$ with $\theta(k)\in \{0,1\}$, we have from \eqref{eq:expan} that  
$$\begin{aligned}
u(x) & = \hat{u}(0) + \sum_{k\in \N_0^d\setminus \{\mathbf{0}\} }  \hat{u}(k)\cdot  \frac{1}{2^d} \sum_{\xi\in \Xi}\cos(\pi k_\xi \cdot x)\\
& =   \hat{u}(0) + \sum_{k\in \N_0^d\setminus \{\mathbf{0}\} }    \left|\hat{u}(k)\right| \text{sign} (\hat{u}(k))\cdot   \frac{1}{2^d} \sum_{\xi\in \Xi}\cos(\pi k_\xi \cdot x)\\
& =  \hat{u}(0) + \sum_{k\in \N_0^d\setminus \{\mathbf{0}\} }    \left|\hat{u}(k)\right| \cdot \frac{1}{2^d} \sum_{\xi\in \Xi}\cos(\pi (k_\xi \cdot x + \theta_k))\\
&=  \hat{u}(0) + \sum_{k\in \N_0^d\setminus \{\mathbf{0}\} }     \frac{1}{Z_u}  \left|\hat{u}(k)\right| (1 + \pi^2|k|_1^2)  \cdot  \frac{Z_u }{1 + \pi^2|k|_1^2} \cdot \frac{1}{2^d} \sum_{\xi\in \Xi}\cos(\pi (k_\xi \cdot x + \theta_k))\\
& =:  \hat{u}(0) + \int g(x, k) \mu(dk),
\end{aligned}
$$
where $\mu(dk)$ is the probability measure  on $\N_0^d \setminus \{\mathbf{0}\}$ defined by 
$$
\mu(dk) = \sum_{k\in\N_0^d\setminus \{\mathbf{0}\} } \frac{1}{Z_u} \big|\hat{u}(k)\big| (1 + \pi^2|k|_1^2) \delta(dk)$$ 
with normalizing constant $Z_u =  \sum_{k\in\N_0^d\setminus \{\mathbf{0}\} }|\hat{u}(k)| (1 + \pi^2|k|_1^2) \leq \|u\|_{\mB(\Omega)}$ and 
$$
g(x, k) = \frac{Z_u}{1 + \pi^2|k|_1^2}\cdot  \frac{1}{2^d} \sum_{\xi\in \Xi}\cos(\pi (k_\xi \cdot x + \theta_k)).
$$
Observe that the function $g(x,k) \in C^2(\Omega)$ for every $k\in \N_0^d \setminus \{\mathbf{0}\}$. Moreover, it is straightforward to show that the following bounds hold:
$$\begin{aligned}
& \|g(\cdot,k)\|_{H^1(\Omega)}  = Z_u \sqrt{\frac{\alpha_k}{1 + \pi^2 |k|_1^2}} \leq \|u\|_{\mB(\Omega)},\\
& \|D^sg(\cdot,k)\|_{L^\infty(\Omega)}  \leq Z_u  \leq \|u\|_{\mB(\Omega)} \text{ for } s = 0,1,2.
\end{aligned} 
$$

Let us define for a constant $B>0$  the function class 
$$
\mF_{\cos}(B)  := \Big\{\frac{\gamma }{1 + \pi^2|k|_1^2} \cos(\pi(k \cdot x + b)), k\in \Z^d \setminus \{\mathbf{0}\}, |\gamma|\leq B, b\in \{0,1\} \Big\}.
$$
It follows from the calculations above that if $u\in \mB(\Omega)$, then $\bar{u} := u - \hat{u}(0)$ lies in the  $H^1$-closure of the convex hull of $\mF_{\cos}(B)$ with $B  = \|u\|_{\mB(\Omega)}$. Indeed, if $\{k^i\}_{i=1}^m$ is an i.i.d. sequence of random samples from the probability measure $\mu$, then it follows from Fubini's theorem that 
$$\begin{aligned}
& \bE \left\|\bar{u}(x)- \frac{1}{m}\sum_{i=1}^m g(x, k^i)\right\|_{H^1(\Omega)}^2\\
& =
\bE  \int_{\Omega} \left|\bar{u}(x)- \frac{1}{m}\sum_{i=1}^m g(x, k^i)\right|^2dx   + \bE \int_{\Omega} \left|\nabla \bar{u}(x)- \frac{1}{m}\sum_{i=1}^m \nabla g(x, k^i)\right|^2dx \\
& = \frac{1}{m} \int_{\Omega} \text{Var}[g(x, k)]dx +  \frac{1}{m} \int_{\Omega} \text{Tr}(\text{Cov}[\nabla g(x, k)])dx\\
& \leq \frac{\bE \|g(\cdot,k)\|_{H^1(\Omega)}^2}{m}\\
& \leq \frac{\|u\|_{\mB(\Omega)}^2}{m}.
\end{aligned}
$$
Therefore the expected $H^1$-norm of an average of $m$ elements in $\mF_{\cos}(B)$ converges to zero as $m\gt \infty$. This in particular implies that there exists a sequence of  convex combinations of points in $\mF_{\cos}(B)$ converging to $\bar{u}$ in $H^1$-norm. 
Since the $H^1$-norm of any function in $\mF_{\cos}(B)$ is bounded by $B$, an application of Maurey's empirical method (see Lemma \ref{lem:convex}) yields the following theorem. 

\begin{theorem}\label{thm:gcos}
 Let $u\in  \mB(\Omega)$. Then there exists $u_m$ which is a convex combination of $m$ functions in $\mF_{\cos}(B)$ with $B = \|u\|_{\mB(\Omega)}$  such that
$$
\|u  - \hat{u}(0) - u_m \|_{H^1(\Omega)}^2 \leq \frac{\|u\|_{\mB(\Omega)}^2}{m}.
$$
\end{theorem}

\begin{lemma} \cite{pisier1981remarques,barron1993universal} \label{lem:convex}
Let $u$ belongs to the closure of the convex hull of a set $\mG$ in a Hilbert space. Let the Hilbert norm of of each element of $\mG$ be bounded $B>0$. Then for every $m\in \N$, there exists $\{g_i\}_{i=1}^m \subset \mG$ and  $\{c_i\}_{i=1}^m \subset [0,1]$ with $\sum_{i=1}^m c_i = 1$ such that 
$$
\Big\|u - \sum_{i=1}^m c_i g_i\Big\|^2 \leq \frac{B^2}{m}. 
$$
\end{lemma}

\subsection{Reduction to ReLU and Softplus Activation Functions}
Notice that every function in $\mF_{\cos}(B)$ is the composition of the one dimensional function $g$ defined on $[-1,1]$ by
\begin{equation}\label{eq:g}
    g(z) = \frac{\gamma }{1 + \pi^2|k|_1^2}\cos(\pi (|k|_1 z + b))
\end{equation}
 with $k\in \Z^d \setminus \{\mathbf{0}\}, |\gamma|\leq B$ and $b \in \{0,1\}$, and a linear function $z =w \cdot x$ with $w = k/|k|_1$.  It is clear that $g\in C^2([-1,1])$ and $g$ satisfies that 
\begin{equation}\label{eq:gbd}
    \|g^{(s)}\|_{L^\infty([-1,1])}  \leq  |\gamma| \leq B  \text{ for } s=0,1,2.
\end{equation}
Since $b\in \{0,1\}$, it also holds that $g^\prime(0)=0$.

\begin{lemma}\label{lem:relu1}
Let $g \in C^2([-1,1])$ with $\|g^{(s)}\|_{L^\infty([-1,1])} \leq B$ for $s=0,1,2$. Assume that $g^\prime(0) = 0$. Let $\{z_j\}_{j=0}^{2m}$ be a partition of  $[-1,1]$ with $z_0=-1, z_m = 0, z_{2m}=1$ and $z_{j+1} - z_j = h = 1/m$ for each $j=0,\cdots,2m-1$.  Then there exists a two-layer ReLU network $g_m$ of the form 
\begin{equation}\label{eq:gm1}
   g_m(z) = c + \sum_{i=1}^{2m} a_i  \relu(\epsilon_i z - b_i), z\in [-1,1] 
\end{equation}
with $c = g(0), b_i\in[-1,1] $ and $\epsilon_i\in\{\pm 1\},i=1,\cdots,2m$  such that 
\begin{equation}\label{eq:w1ebd}
    \|g - g_m\|_{W^{1,\infty}([-1,1])} \leq \frac{2B}{m}.
\end{equation}
Moreover, we have that 
$
|a_i| \leq \frac{2 B}{m} $ and that $
|c| \leq B.
$
\end{lemma}
\begin{proof}

Let $g_m$ be the piecewise linear interpolation of $g$ with respect to the grid $\{z_j\}_{j=0}^{2m}$, i.e.  
$$
g_m(z) = g(z_{j+1}) \frac{z-z_{j}}{h} +  g(z_{j}) \frac{z_{j+1}-z}{h} \text{ if }  z\in [z_{j}, z_{j+1}].
$$
According to \cite[Chapter 11]{ascher2011first}, 
$$
\|g - g_m\|_{L^\infty([-1,1])} \leq \frac{h^2}{8} \|g^{\prime\prime}\|_{L^\infty([-1,1])}.
$$
Moreover, $
\|g^\prime - g_m^\prime\|_{L^\infty([-1,1])} \leq h\|g^{\prime\prime}\|_{L^\infty([-1,1])}.
$
In fact, consider $z\in [z_{j},z_{j+1}]$ for some $j\in\{0,\cdots,2m-1\}$. By the mean value theorem, there exist $\xi, \eta \in (z_j,z_{j+1})$ such that  $(g(z_{j+1} - g(z_j)))/h = g^\prime(\xi)$ and hence 
$$\begin{aligned}
\Big|g^\prime(z) - \frac{g(z_{j+1}) - g(z_i)}{h}\Big| & = \Big|g^\prime(z) - g^\prime(\xi)\Big| \\
& = |g^{\prime\prime}(\eta)| |z-\xi| \\
& \leq h\|g^{\prime\prime}\|_{L^\infty([-1,1])}.
\end{aligned}
$$
This proves the error bound \eqref{eq:w1ebd}. 

Next, we show that $g_m$ can be represented by a two-layer ReLU neural network. 
Indeed, it is easy to verify that $g_m$ can be rewritten as
\begin{equation}\label{eq:gm2}
    g_m(z) = c + \sum_{i=1}^{m} a_i \relu( z_{i}-z)   +  \sum_{i={m+1}}^{2m} a_i  \relu(z - z_{i-1})  , z\in [-1,1],
\end{equation}
where $c =  g(z_m) = g(0)$ and the parameters $a_i$   defined by 
$$
a_i = \begin{cases}
\frac{g(z_{m+1})-g(z_m)}{h}, & \text{ if } i =m+1,\\
\frac{g(z_{m-1})-g(z_m)}{h}, & \text{ if } i =m,\\
\frac{g(z_i)-2g(z_{i-1}) + g(z_{i-2})}{h}, & \text{ if } i> m+1,\\
\frac{g(z_{i-1})-2g(z_{i}) + g(z_{i+1})}{h}, & \text{ if } i< m.
\end{cases}
$$
Furthermore, by again the mean value theorem, there exists $\xi_1, \xi_2\in (z_{m},z_{m+1})$ such that
$
|a_{m+1}| = |g^\prime(\xi_1)| =|g^\prime(\xi_1) - g^\prime(0)| = |g^{\prime\prime}(\xi_2) \xi_1| \leq B h.
$
In a similar manner one can obtain  that $|a_{m}|\leq B h $ and $|a_{i}|\leq 2B h$ if $i \notin \{m,m+1\}$.

Finally, by setting $\epsilon_i = -1, b_i = -z_i$ for $i=1,\cdots,m$ and $\epsilon_i = 1, b_i = z_{i-1}$ for $i=m+1,\cdots,2m$, one obtains the desired form \eqref{eq:gm1} of $g_m$.
This completes the proof of the lemma. 
\end{proof}

The following proposition is a direct consequence of Lemma \ref{lem:relu1}.
\begin{proposition}\label{prop:grelu}
  Define the function class 
$$
\mF_{\relu} (B) :=\Big\{ c + \gamma \relu(w\cdot x - t), |c|\leq 2B , |w|_1=1,  |t|\leq 1, |\gamma|\leq 4B \}.
$$
 Then for any constant $\tilde{c}$ such that $|\tilde{c}|\leq B$, the set 
 $
 \tilde{c} + \mF_{\cos}(B)
 $
 is in the $H^1$-closure of the convex hull of $\mF_{\relu}(B)$. 
\end{proposition}

\begin{proof}First  Lemma \ref{lem:relu1} states that each $C^2$-function $g$ with $g^\prime(0)=0$ and with up to second order derivatives bounded by $B$ can be well approximated in $H^1$-norm by a linear combination of a constant function  and the ReLU functions $\relu(\epsilon z-t)$ with the sum of the absolute values of the combination coefficients bounded by $4B$. As a result, the function $g$ defined in \eqref{eq:g} lies in the closure of the convex hull of functions $c + \gamma \relu(\epsilon z-t)$ with $|c|\leq B, |\gamma|\leq 4B , |t|\leq 1$. Then the proposition follows from absorbing the additive constant $\tilde{c}$ into the constant $c$ in the definition of $\mF_{\relu} (B)$. 
\end{proof}

With Proposition \ref{prop:grelu}, we are ready to give the proof of Theorem \ref{thm:relu}.
\begin{proof}[Proof of Theorem \ref{thm:relu}]
    Observe that if $u\in \mF_{\relu}(B)$, then  
$$
\|u\|_{H^1(\Omega)}^2 
\leq (c + 2\gamma)^2 + \gamma^2 \leq (10^2+4^2) B^2 = 116 B^2. 
$$
Therefore Theorem \ref{thm:relu} follows directly from  Lemma \ref{lem:convex}, Proposition \ref{prop:grelu} with $\tilde{c} = \hat{u}(0)$ and the fact that $|\hat{u}(0)|\leq \|u\|_{\mB(\Omega)}$.
\end{proof}

Next we proceed to prove Theorem \ref{thm:sp} which concerns approximating spectral Barron functions using two-layer networks with the Softplus activation. To this end, let us first state a lemma which shows that $\relu$ can be well approximated by $\sp_\tau$ for  $\tau \gg 1$.

\begin{lemma}\label{lem:relusp}
The following inequalities hold:
\begin{align*}
\text{(i)} \hspace{4em}  && 
  |\relu(z) - \sp_{\tau}(z)| & \leq \frac{1}{\tau}  e^{-\tau |z|}, \; \forall z\in [-2,2]; \\
\text{(ii)} \hspace{4em} && |\relu^\prime(z) - \sp_{\tau}^\prime(z)| & \leq e^{-\tau |z|}, \; \forall z\in [-2,0)\cup(0,2]; \\
\text{(iii)} \hspace{4em} &&\|\sp_{\tau}\|_{W^{1,\infty}([-2,2])} & \leq 3 + \frac{1}{\tau}.
\end{align*}

\end{lemma}
\begin{proof}
Notice that 
$
\relu(z) - \sp_\tau(z) = -\frac{1}{\tau} \ln (1 + e^{-\tau |z|})
$. Hence inequality (i) follows from that 
$$
|\relu(z) - \sp_\tau(z) |\leq \frac{1}{\tau}  \ln (1 + e^{-\tau |z|}) \leq \frac{e^{-\tau |z|}}{\tau},
$$
where the second inequality follows from the simple inequality $\ln(1+x) \leq x$ for $x>-1$. In addition, inequality (ii) holds since 
$$
|\relu^\prime(z) - \sp_\tau^\prime(z)| = \Big|\frac{1}{1+ e^{\tau |z|}}\Big| \leq  e^{-\tau |z|}, \text{ if } z\neq 0.
$$
Finally,  inequality (iii) follows from that 
$$
\|\sp_{\tau}(z)\|_{L^\infty([-2,2])}  = \sp_\tau(2) \leq 2 + \frac{1}{\tau}
$$ 
and that 
\begin{equation*}
|\sp^\prime_{\tau}(z)| = \Big|\frac{1}{1 + e^{\tau z}}\Big| \leq 1. \qedhere
\end{equation*}
\end{proof}

\begin{lemma}\label{lem:sp1}
Let $g \in C^2([-1,1])$ with $\|g^{(s)}\|_{L^\infty([-1,1])} \leq B$ for $s=0,1,2$. Assume that $g^\prime(0) = 0$. Let $\{z_j\}_{j=-m}^m$ be a partition of  $[-1,1]$ with $m\geq 2$ and $z_{-m}=-1, z_0=0, z_m=1$ and $z_{j+1} - z_j = h = 1/m$ for each $j=-m,\cdots,m-1$.  Then there exists a two-layer neural network $g_{\tau,m}$  of the form 
\begin{equation}\label{eq:gtm}
g_{\tau,m}(z) = c + \sum_{i=1}^{2m} a_i  \sp_{\tau}(\epsilon_i z - b_i), z\in [-1,1]
\end{equation}
with $c = g(0)\leq B, b_i\in[-1,1], |a_i|\leq 2B/m $ and $\epsilon_i\in\{\pm 1\},i=1,\cdots,2m$  such that 
\begin{equation}\label{eq:w1ebd2}
    \|g - g_{\tau,m}\|_{W^{1,\infty}([-1,1])} \leq 6B \delta_\tau,
\end{equation}
where 
\begin{equation}\label{eq:deltatau}
  \delta_\tau := \frac{1}{\tau}\Big(1 + \frac{1}{\tau}\Big)\Big(\log\big(\frac{\tau}{3}\big)+1\Big).
\end{equation}

\end{lemma}

\begin{proof}
Thanks to Lemma \ref{lem:relu1}, there exists $g_m$ of the form 
\begin{equation}\label{eq:gm3}
    g_m(z) = c + \sum_{i=1}^{m} a_i \relu( z_{i}-z)   +  \sum_{i={m+1}}^{2m} a_i  \relu(z - z_{i-1})  , z\in [-1,1]
\end{equation}
such that $\|g- g_m\|_{W^{1,\infty}([-1,1])}\leq 2B/m $. More importantly, the coefficients $a_i$ satisfies that $|a_i|\leq 2B/m$ so that $\sum_{i=1}^{2m} a_i \leq 4B$. Now let $g_{\tau,m}$ be the function obtained by replacing the activation $\relu$ in $g_m$  by $\sp_{\tau}$, i.e.
\begin{equation}\label{eq:gmt1}
    g_{\tau,m}(z) = c + \sum_{i=1}^{m} a_i \sp_{\tau}( z_{i}-z)   +  \sum_{i={m+1}}^{2m} a_i  \sp_{\tau}(z - z_{i-1})  , z\in [-1,1].
\end{equation}
Suppose that $z\in (z_j, z_{j+1})$ for some fixed $j< m-1$. Then thanks to  Lemma \ref{lem:relusp} - (i), the bound $|a_i|\leq 2B/m$ and  the fact that $|z_i - z|\geq 1/m$ if $i\neq j$ while $z\in (z_j,z_{j+1})$, we have
$$\begin{aligned}
|g_m(z) - g_{\tau, m}(z)| & \leq |a_{j}| |\relu(z_j-z) - \sp_{\tau}(z_j-z)| \\
& \qquad + \sum_{i=1,i\neq j}^m |a_{i}| |\relu(z_i-z) - \sp_{\tau}(z_i-z)| \\
& \qquad +\sum_{i={m+1}}^{2m} |a_i|  |\relu(z - z_{i-1}) - \sp_{\tau}(z - z_{i-1}) |\\
& \leq \frac{2B}{m \tau} + \frac{2B}{\tau} e^{-\tau |x|} \mathbf{1}_{|x|\geq 1/m}.
\end{aligned}
$$
Similar bounds hold for the case where $z\in (z_j, z_{j+1})$ for $j> m$. Lastly, if $z\in (z_m, z_{m+1})$, then both the $m$-th and $m+1$-th term in \eqref{eq:gm3} and \eqref{eq:gmt1} depend on $z_m$, from which we get 
$$
|g_m(z) - g_{\tau, m}(z)| \leq  \frac{4B}{m \tau} + \frac{2B}{\tau} e^{-\tau |x|} \mathbf{1}_{|x|\geq 1/m}.
$$
Therefore we have obtained that 
$$
\|g_m - g_{\tau, m}\|_{L^\infty([-1,1])} \leq \frac{4B}{m \tau} + \frac{2B}{\tau} e^{-\tau |x|} \mathbf{1}_{|x|\geq 1/m}.
$$
Thanks to  Lemma \ref{lem:relusp} - (ii), the same argument carries over to the estimate for the difference of the derivatives and leads to 
$$
\|g^\prime_m - g^\prime_{\tau, m}\|_{L^\infty([-1,1])}  \leq
 \frac{4B}{m} + 2B  e^{-\tau |x|} \mathbf{1}_{|x|\geq 1/m}.
$$
Combining the estimates above with that $\|g - g_m\|_{W^{1,\infty}([-1,1])} \leq 2B/m$ yields that
$$\begin{aligned}
\|g - g_{\tau, m}\|_{W^{1,\infty}([-1,1])} & \leq
\|g - g_{m}\|_{W^{1,\infty}([-1,1])}+ 
\|g_m - g_{\tau, m}\|_{W^{1,\infty}([-1,1])}\\
& \leq \frac{2B}{m} + \frac{4B}{m \tau} + \frac{2B}{\tau} e^{-\tau |x|} \mathbf{1}_{|x|\geq 1/m}\\
& \leq 2B \Big( 1+ \frac{1}{\tau}\Big) \Big(\frac{3}{m} + e^{-\frac{\tau}{m}}\Big)\\
& = 6B\delta_{\tau}.
\end{aligned}
$$
We have used the fact that $\max_{0<x\leq 1/2} 3x + e^{-\tau x} = \bigl(\log\big(\frac{\tau}{3}\big)+1\bigr)\frac{3}{\tau}$ in the last inequality. 
   The proof of the lemma is finished by combining the estimates above and by rewriting \eqref{eq:gmt1} in the form of \eqref{eq:gtm}.
\end{proof}

Now we are ready to present the proof of Theorem \ref{thm:sp}. To do this, let us define  the function class 
$$
\mF_{\sp_{\tau}} (B) :=\Big\{ c + \gamma \sp_{\tau}(w\cdot x - t), |c|\leq 2B , |w|_1=1,  |t|\leq 1, |\gamma|\leq 4B \Big\}.
$$
Note by (iii) of Lemma \ref{lem:relusp} that 
 \begin{equation}\label{eq:h1gsp}
\sup_{u\in \mF_{\sp_{\tau}} (B)}  \|f\|_{H^1(\Omega)} \leq 2B + 4B\|\sp_{\tau}\|_{W^{1,\infty}([-2,2])}
\leq 14B + \frac{4B}{\tau}.
 \end{equation}

\begin{proof}[Proof of Theorem \ref{thm:sp}]
First according to Theorem \ref{thm:gcos}, $u - \hat{u}(0)$ lies in the $H^1$-closure of the convex hull of  $\mF_{\cos}(B)$ with $B=\|u\|_{\mB(\Omega)}$. Note that each  function in $\mF_{\cos}(B)$ is a composition of the multivariate linear function $z = w \cdot x$ with $|w|= 1$ and the univariate function $g(z)$ defined in \eqref{eq:g} such that $g^\prime(0)=0$ and $\|g^{(s)}\|_{L^\infty([-1,1])}\leq B$ for $s=0,1,2$. By Lemma \ref{lem:sp1},  such $g$ can be approximated by $g_{\tau,m}$ which lies in the convex hull of the set of functions
$$
\Big\{ c + \gamma \sp_{\tau}(\epsilon z-b), |c|\leq B, \epsilon\in \{\pm 1\}, |b|\leq 1, \gamma \leq 4B \Big\}. 
$$
Moreover, 
 $ \|g - g_{\tau,m}\|_{W^{1,\infty}([-1,1])} \leq 6B \delta_\tau$.  As a result, we have that 
$$
\|g(w\cdot x) - g_{\tau,m}(w\cdot x)\|_{H^1(\Omega)} \leq
\|g - g_{\tau,m}\|_{W^{1,\infty}([-1,1])} \leq 6B \delta_\tau.
$$
This combining with the fact that $|\hat{u}(0)| \leq B$ yields that there exists a function $u_{\tau}$
in the closure of the convex hull of $\mF_{\sp_{\tau}}(B)$ such that 
$$\|u  - u_{\tau}\|_{H^1(\Omega)}\leq 6B \delta_{\tau}.$$
Thanks to Lemma \ref{lem:convex} and the bound \eqref{eq:h1gsp}, there exists $u_m\in \mF_{\sp_{\tau},m}(B)$,  which is a convex combination of $m$ functions in $\mF_{\sp_{\tau}}(B)$ such that $$
\|u_{\tau}-u_m\|_{H^1(\Omega)} \leq \frac{B \Big(\frac{4}{\tau} + 14\Big)}{\sqrt{m}}.
$$
Combining the last two inequalities leads to  
$$
\|u  - u_m\|_{H^1(\Omega)} \leq 6B  \delta_{\tau} + \frac{B \Big(\frac{4}{\tau} + 14\Big)}{\sqrt{m}}.
$$
Setting $\tau = \sqrt{m}\geq 1$ and using \eqref{eq:deltatau}, we obtain that
$$\begin{aligned}
  \|u - u_m\|_{H^1(\Omega)} 
& \leq \frac{6B}{\tau} \Big(1 + \frac{1}{\tau} \Big)  \Big(\log  \Big(\frac{\tau}{3} \Big) +1 \Big) + \frac{B}{\sqrt{m}}  \Big(\frac{4}{\tau} + 14 \Big)
\\
& \leq \frac{6B}{\sqrt{m}} 2  \Big(\frac{1}{2} \log(m) +1 \Big) + \frac{18B}{\sqrt{m}}\\
& =   \frac{B(6\log(m)+30)}{\sqrt{m}}.
\end{aligned}
$$
This proves the desired estimate. 
\end{proof}

\section{Rademacher complexities of  two-layer neural networks}
The goal of this section is to derive the Rademacher complexity bounds for some two-layer neural-network function classes that are relevant to the Ritz losses of the Poisson and the static Schr\"odinger equations. These bounds will be essential for obtaining the generalization bounds in Theorem \ref{thm:main1} and Theorem \ref{thm:main2}. 

First let us consider for fixed positive constants $C, \Gamma, W$ and $T$ 
the set of two-layer neural networks 
\begin{equation}\label{eq:fm}
\begin{aligned}
  \mF_m & = \Big\{u_{\theta}(x)= c + \sum_{i=1}^m \gamma_i \phi(w_i \cdot x + t_i ),\ x\in \Omega, \theta\in \Theta \ \big|\  |c|\leq C, \sum_{i=1}^m |\gamma_i| \leq \Gamma, \\
 & \qquad \qquad  |w_i|_1 \leq W, |t_i| \leq T\Big\}.
\end{aligned} 
\end{equation}
Here $\phi$ is the activation function, $\theta = (c,\{\gamma_i\}_{i=1}^m,\{w_i\}_{i=1}^m, \{t_i\}_{i=1}^m )$ denotes collectively the parameters of the two-layer neural network, $\Theta = \Theta_c\times \Theta_\gamma \times \Theta_{w}\times\Theta_{t} =  [-C,C]\times B_{1}^m(\Gamma)\times \bigl( B_{1}^d(W) \bigr)^m\times [-T,T]^m$ represents the parameter space. We shall consider the set $\Theta$ endowed with the  metric $\rho$ defined  for $\theta = (c, \gamma, w, t), \theta^\prime =(c^\prime, \gamma^\prime, w^\prime, t^\prime)\in \Theta$ by
\begin{equation}\label{rhotheta}
\rho_{\Theta}(\theta, \theta^\prime) = \max \{|c-c^\prime|, |\gamma - \gamma^\prime|_1, \max_i |w_i-w_i^\prime|_1, \|t-t^\prime\|_\infty\}. 
\end{equation} 
Throughout the section we assume that $\phi$ satisfies the following assumption, which particularly holds for the Softplus activation function.
\begin{assum}\label{assm:phi2}
$\phi\in C^2(\R)$ and that $\phi$ (resp. $\phi^\prime$, the derivative of $\phi$) is $L$-Lipschitz (resp. is $L^\prime$-Lipschitz) for some $L,L^\prime>0$. Moreover, there exist positive constants $\phi_{\max}$ and $\phi^\prime_{\max}$ such that 
$$
\sup_{w\in \Theta_w,t\in \Theta_t,x\in \Omega} |\phi(w\cdot x+t)|\leq \phi_{\max} \text{ and } \sup_{w\in \Theta_w,t\in \Theta_t,x\in \Omega} |\phi^\prime(w\cdot x+t)|\leq \phi^\prime_{\max}.
$$
\end{assum}

Recall that the  Rademacher complexity of a function class $\mG$ is defined by 
$$
R_n (\mG) = \bE_{Z} \bE_{\sigma} \Big[\sup_{g\in \mG} \Big| \frac{1}{n}\sum_{j=1}^n \sigma_j g(Z_j) \Big| \; \Big| \; Z_1, \cdots, Z_n\Big].
$$
In the subsequent proof, it will be useful to use the following modified Rademacher complexity $\tilde{R}_n(G)$ without the absolute value sign:
$$
\tilde{R}_n (\mG) = \bE_{Z} \bE_{\sigma} \Big[\sup_{g\in \mG}  \frac{1}{n}\sum_{j=1}^n \sigma_j g(Z_j) \; \Big| \; Z_1, \cdots, Z_n\Big].
$$

The lemma below bounds the Rademacher complexity of $\mF_m$. 

\begin{lemma}\label{lem:rad1}
Assume that the activation function $\phi$  is $L$-Lipschitz. Then
  $$
  R_n(\mF_m) \leq \frac{4\Gamma L(W \sqrt{d}+ T) + 2\Gamma^2 |\phi(0)|}{\sqrt{n}}.
$$
\end{lemma}
\begin{proof} Let $\bar{\phi}(x) = \phi(x) - \phi(0)$.
 First observe that
  \begin{align*}
& \bE_{\sigma} \Big[\sup_{f\in \mF_m}  \frac{1}{n}\sum_{j=1}^n \sigma_j f(Z_j) \Big| Z_1, \cdots, Z_n\Big]\\
  & = \bE_{\sigma} \Big[\sup_{\Theta}  \frac{1}{n}\sum_{j=1}^n\sigma_j \big(c + \sum_{i=1}^m \gamma_i \phi(w_i \cdot Z_j + t_i )\big) \Big| Z_1, \cdots, Z_n\Big]\\
  & = \bE_{\sigma} \Big[\sup_{\Theta}  \frac{1}{n}\sum_{j=1}^n\sigma_j \sum_{i=1}^m \gamma_i \phi(w_i \cdot Z_j + t_i ) \Big| Z_1, \cdots, Z_n\Big]\\
  & \leq \frac{1}{n} \bE_{\sigma} \Big[\sup_{\Theta}   \sum_{i=1}^m \gamma_i\sum_{j=1}^n\sigma_j  \bar{\phi}(w_i \cdot Z_j + t_i ) \Big| Z_1, \cdots, Z_n\Big] + \frac{1}{n} \bE_{\sigma} \Big[\sup_{\Theta}   \sum_{i=1}^m \gamma_i\sum_{j=1}^n\sigma_j \phi(0) \Big]\\
  & =: J_1+J_2.\end{align*}
  Using the fact that $\bar{\phi}(\cdot) = \phi(\cdot) - \phi(0)$ is $L$-Lipschitz, one has that 
    \begin{align*}|
    J_1
  & \leq \frac{1}{n} \sum_{i=1}^m |\gamma_i|\cdot  \bE_\sigma\Big[\sup_{|w|_1\leq W, |t|\leq T}\Big|\sum_{j=1}^n\sigma_j \bar{\phi}(w \cdot Z_j + t)\Big| \;\Big|\; Z_1, \cdots, Z_n \Big] \\
  &\leq \frac{2\Gamma L}{n} \Big( \bE_\sigma\Big[\sup_{|w|_1\leq W}\Big|\sum_{j=1}^n\sigma_j w \cdot Z_j\Big| \;\Big|\; Z_1, \cdots, Z_n \Big] + \bE_\sigma\Big[\sup_{|t|\leq T}\Big| \sum_{j=1}^n \sigma_j t\Big| \Big]\Big)\\
  & \leq \frac{2\Gamma L}{n}  \Big(W \cdot \bE_\sigma \Big|\sum_{j=1}^n \sigma_j Z_j\Big| + T \bE_\sigma \Big[\Big|\sum_{j=1}^n \sigma_j\Big|\Big]\Big)\\
  & \leq \frac{2\Gamma L}{n} \Big(W \cdot \sqrt{\sum_{j=1}^n |Z_j|^2} + T \cdot \sqrt{\bE_{\sigma} \Big[\sum_{j=1}^n \sigma_j^2}\Big]\Big)\\
  & \leq \frac{2\Gamma L(W \sqrt{d}+ T)}{\sqrt{n}}.
  \end{align*}
Note that in the second inequality we have used the Talagrand's contraction principle (Lemma \ref{lem:Talagrand} below). Moreover, since $\sum_{i=1}^m |\gamma_i| \leq \Gamma$,  it is easy to see that 
 \begin{align*}
    J_2 
  & \leq \frac{\Gamma |\phi(0)|}{n} \bE_\sigma \Big[\Big|\sum_{j=1}^n \sigma_j\Big|\Big]\\
  & \leq \frac{\Gamma |\phi(0)|}{n} \sqrt{\bE_{\sigma} \Big[\sum_{j=1}^n \sigma_j^2}\Big]\\
  & = \frac{\Gamma |\phi(0)|}{\sqrt{n}}.
 \end{align*}
Combining the estimates above and then taking the expectation w.r.t.{} $Z_j$ yields that $\tilde{R}_n(\mF_m)\leq \frac{2\Gamma L(W\sqrt{d}+ T) + \Gamma|\phi(0)|}{\sqrt{n}}$. This combined with Lemma  \ref{lem:rn1} below leads to the desired estimate.
\end{proof}
\begin{lemma}[{Ledoux-Talagrand  contraction \cite[Theorem 4.12]{ledoux1991probability}}]\label{lem:Talagrand}
  Assume that $\phi:\R\gt\R$ is $L$-Lipschitz with $\phi(0)=0$. Let $\{\sigma_i\}_{i=1}^n$ be independent
Rademacher random variables. Then for any $T\subset \R^n$
$$
\bE_\sigma \sup_{(t_1,\cdots,t_n)\in T} \Big|\sum_{i=1}^n \sigma_i \phi(t_i)\Big| \leq 2L\cdot   \bE_\sigma \sup_{(t_1,\cdots,t_n)\in T} \Big|\sum_{i=1}^n \sigma_i t_i \Big| .
$$
\end{lemma}

\begin{lemma}\cite[Lemma 1]{tengyunotes}\label{lem:rn1}
  Assume that the set of functions $\mG$ contains the zero function. Then 
  $$ 
  R_n(\mG) \leq 2\tilde{R}_n(\mG).
$$
\end{lemma}

Recall the sets  of two-layer neural networks $\mF_{\relu,m}(B)$  and $\mF_{\sp_\tau,m}(B)$ defined by \eqref{eq:frelum} and \eqref{eq:fspm} respectively. Since both $\relu$ and $\sp_\tau$ are $1$-Lipschitz and $\relu(0)=0$, $\sp_\tau(0) = \frac{\ln 2}{\tau}$, the following corollary is a direct consequence of  Lemma~\ref{lem:rad1}. 

\begin{corollary}\label{cor:rnfm}
  $$
  R_n(\mF_{\relu,m}(B))\leq \frac{16(\sqrt{d}+1) B}{\sqrt{n}} \quad \text{ and } \quad  R_n(\mF_{\sp_\tau,m}(B))\leq \frac{16 (\sqrt{d}+1 + \frac{2\ln 2}{\tau}) B}{\sqrt{n}}.
  $$
\end{corollary}

Given the source function $f\in L^\infty(\Omega)$ and the potential $V\in L^\infty(\Omega)$, we recall the function classes associated to the Ritz losses of Poisson equation and the static Schr\"odinger equation
\begin{equation}\label{eq:gmps}
\begin{aligned}
\mG_{m,P}  & := \Big\{g: \Omega \gt \R \ \big|\  g = \frac{1}{2}  | \nabla u|^2 - f u  \text{ where }u \in \mF_m\Big\},\\
\mG_{m,S} & :=  \Big\{g: \Omega \gt \R \ \big|\  g = \frac{1}{2}  | \nabla u|^2 + \frac{1}{2} V |u|^2 - f u  \text{ where }u \in \mF_m\Big\}.
\end{aligned}
\end{equation}
In the sequel we aim to bound the Rademacher complexities of $\mG_{m,P} $ and $\mG_{m,S} $ defined above. This will be achieved by bounding the  Rademacher complexities of the following function classes
 $$
 \begin{aligned}
         \mG_m^1 & := \Big\{g: \Omega \gt \R \ \big|\  g = \frac{1}{2}  | \nabla u|^2 \text{ where }u \in \mF_m\Big\},\\
\mG_m^2  & := \Big\{g: \Omega \gt \R \ \big|\  g = f u  \text{ where }u \in \mF_m\Big\},\\
\mG_m^3  & := \Big\{g: \Omega \gt \R \ \big|\  g = \frac{1}{2} V |u|^2  \text{ where }u \in \mF_m\Big\}.
 \end{aligned}
 $$
The celebrated Dudley's theorem will be used to bound the Rademacher complexity in terms of the metric entropy. To this end, let us first recall the metric entropy and the Dudley's theorem below. 

Let $(E,\rho)$ be a metric space with metric $\rho$. A {\em $\delta$-cover} of a set $A\subset E$ with respect to $\rho$ is  a collection of points $\{x_1,\cdots, x_n\}\subset A$ such that for every $x\in A$, there exists $i\in \{1,\cdots,n\}$ such that $\rho(x,x_i)\leq \delta$. The $\delta$-covering number 
 $\mN(\delta,A,\rho)$ is the  cardinality of the smallest $\delta$-cover of the set $A$ with respect to the metric $\rho$. Equivalently, the $\delta$-covering number 
 $\mN(\delta,A,\rho)$   is the minimal number of balls $B_{\rho}(x,\delta)$ of radius $\delta$ needed to cover the set $A$.

\begin{theorem}[Dudley's theorem] \label{thm:dudley} Let $\mF$ be a function class such that $\sup_{f\in \mF}\|f\|_{\infty}\leq M$. Then the Rademacher complexity $R_n(\mF)$ satisfies that 
$$
R_n(\mF) \leq \inf_{0\leq \delta\leq M} \Big\{4\delta + \frac{12}{\sqrt{n}}\int_\delta^M \sqrt{\log \mN(\eps,\mF, \|\cdot\|_\infty)} \,d\eps\Big\}.   
$$
\end{theorem} 
Note that our statement of Dudley's theorem is slightly different from the standard Dudley's theorem where the covering number is based on the empirical $\ell^2$-metric instead of the $L^\infty$-metric above. However, since $L^\infty$-metric is stronger than the empirical $\ell^2$-metric and since the covering number is monotonically  increasing with respect to the metric, Theorem~\ref{thm:dudley} follows directly from the classical Dudley's theorem (see e.g.{} \cite[Theorem 1.19]{wolfnotes}).

Let us now state an elementary lemma on the covering number of product spaces.  
\begin{lemma}\label{lem:coverprod}
  Let $(E_i, \rho_i)$ be metric spaces with metrics $\rho_i$ and let $A_i\subset E_i, i=1,\cdots, n$. 
  Consider the product space $E = \times_{i=1}^n E_i$ equipped with the metric $\rho  = \max_i \rho_i$ and the set $A = \times_{i=1}^n A_i$. Then for any $\delta>0$, 
\begin{equation}\label{eq:cover0}
\mN(\delta, A, \rho) \leq \prod_{i=1}^n \mN(\delta, A_i, \rho_i).
\end{equation}
\end{lemma}
\begin{proof}
It suffices to prove the lemma in the case that $n=2$, i.e.,
\begin{equation}\label{eq:cover1}
\mN(\delta, A_1\times A_2, \rho) \leq  \mN(\delta, A_1, \rho_1)\cdot  \mN(\delta, A_2, \rho_2) .
\end{equation}
 Indeed, suppose that $C_1$ and $C_2$ are $\delta$-covers of $A_1$ and $A_2$ respectively. Then it is straightforward that the product set $C_1\times C_2$ is also a $\delta$-cover of $A_1\times A_2$ in the space $(E_1\times E_2,\rho)$ with $\rho = \max (\rho_1,\rho_2)$. Hence $\mN(\delta, A_1\times A_2, \rho)\leq \text{ card} (C_1)\cdot  \text{ card} (C_2)$. Applying this inequality for $C_i$ with  $\text{ card} (C_i) = \mN(\delta, A_i, \rho_i),i=1,2$, we obtain \eqref{eq:cover1}. The general inequality \eqref{eq:cover0} follows by iterating \eqref{eq:cover1}.
\end{proof}

As a consequence of Lemma \ref{lem:coverprod}, the following proposition gives an upper bound for the  covering number $\mN(\delta,\Theta, \rho_{\Theta})$. 
\begin{proposition}\label{prop:covtheta}
  Consider the metric space $(\Theta,\rho_{\Theta})$ with $\rho_\Theta$ defined in \eqref{rhotheta}.  Then for any $\delta>0$, the covering number $\mN(\delta,\Theta, \rho_\Theta)$ satisfies that 
$$
\mN(\delta,\Theta, \rho_\Theta)  
\leq \frac{2C}{\delta}\cdot \Big(\frac{3\Gamma}{\delta}\Big)^m  \cdot\Big(\frac{3W}{\delta}\Big)^{dm}\cdot  \Big(\frac{3T}{\delta}\Big)^m.
$$
\end{proposition}

\begin{proof}
    Thanks to Lemma \ref{lem:coverprod}, 
$$
\begin{aligned}
\mN(\delta, \Theta, \rho) & \leq   \mN(\delta, \Theta_c, |\cdot|)\cdot  \mN(\delta, \Theta_\gamma, |\cdot|_1) \cdot \Big(\mN(\delta, B_{1}^d(W), |\cdot|_1)\Big)^m \cdot \mN(\delta, \Theta_t, |\cdot|_\infty)\\
& \leq \frac{2C}{\delta}\cdot \Big(\frac{3\Gamma}{\delta}\Big)^m  \cdot\Big(\frac{3W}{\delta}\Big)^{dm}\cdot  \Big(\frac{3T}{\delta}\Big)^m,
\end{aligned}
$$
where in the last inequality we have used the fact that the covering number of a $d$-dimensional $\ell^p$-ball of radius $r$ satisfies that 
\begin{equation*}
\mN(\delta, B^d_p(r), |\cdot|_p) \leq \Big(\frac{3r}{\delta}\Big)^d. \qedhere
\end{equation*}
\end{proof}

\smallskip

\subsection*{Bounding $R_n(\mG^1_m)$.}
We would like to  bound $R_n(\mG^1_m)$ from above using metric entropy.
To this end, let us first bound the covering number $\mN(\delta,\mG^1_m, \|\cdot\|_{\infty})$. Recall the parameters $C, \Gamma, W$ and $T$ in \eqref{eq:fm}. With those parameters fixed, to simplify expressions, we introduce the following functions to be used in the sequel 
\begin{align}
    \mM (\delta,\Lambda,m, d) &:= \frac{2C\Lambda}{\delta}\cdot \Big(\frac{3\Gamma\Lambda}{\delta}\Big)^m  \cdot\Big(\frac{3W\Lambda}{\delta}\Big)^{dm}\cdot  \Big(\frac{3T\Lambda}{\delta}\Big)^m, \label{eq:M}\\
    \mZ (M,\Lambda,d) &:=M\big(\sqrt{(\log(2C\Lambda))_+} + \sqrt{(\log(3\Gamma \Lambda) + d\log(3W\Lambda)+ \log(3T\Lambda))_+}\big)\label{eq:Z} \\
    &  \qquad + \sqrt{d+3} \int_0^{M} \sqrt{(\log(1/\eps))_+}d\eps  \nonumber. 
\end{align}


\begin{lemma}\label{lem:coverG1}
Let the activation function $\phi$ satisfy Assumption \ref{assm:phi2}. Then we have 
\begin{equation}\label{eq:covG1}
\mN(\delta,\mG^1_{m}, \|\cdot\|_\infty)\leq \mM(\delta, \Lambda_1, m, d), 
\end{equation}
  where the constant $\Lambda_1$ is defined by 
  \begin{equation}\label{eq:Lambda21}
\Lambda_1 = \Big((W+\Gamma)\phi^\prime_{\max} + 2\Gamma W L^\prime \Big) \Gamma W \phi^{\prime}_{\max}. 
  \end{equation}

\end{lemma}

\begin{proof}
Thanks to Assumption \ref{assm:phi2},  
$
\sup_{\theta\in \Theta}|\phi^\prime(w \cdot x + t)|\leq \phi^\prime_{\max}.
$ 
This implies that 
$$\begin{aligned}
\max_{\theta\in \Theta} |\nabla u_\theta(x)| & \leq \sum_{i=1}^m |\gamma_i| |w_i| |\phi^\prime(w_i \cdot x + t_i)|\\
& \leq \Gamma W \phi^{\prime}_{\max}.
\end{aligned}
$$
Furthermore, for $\theta, \theta^\prime \in \Theta$, by adding and subtracting terms, we have that 
$$
\begin{aligned}
&  |\nabla u_{\theta}(x) - \nabla u_{\theta^\prime}(x)| \leq  \sum_{i=1}^m |\gamma_i - \gamma^\prime_i| |w_i| |\phi^\prime(w_i \cdot x+t_i)|  \\
& + 
\sum_{i=1}^m |\gamma^\prime_i| |w_i-w_i^\prime| |\phi^\prime(w_i \cdot x+t_i)| + 
\sum_{i=1}^m |\gamma^\prime_i|  |w^\prime_i| 
|\phi^\prime(w_i \cdot x+t_i) - \phi^\prime(w^\prime_i \cdot x+t^\prime_i)|\\
& \leq W \phi^\prime_{\max} |\gamma-\gamma^\prime|_1 + \Gamma \phi^\prime_{\max} \max_i |w_i-w^\prime_i| + \Gamma W L^\prime (\max_i |w_i-w^\prime_i|_1+ |t-t^\prime|_\infty)\\
& \leq \Big((W+\Gamma)\phi^\prime_{\max} + 2\Gamma W L^\prime \Big) \rho_{\Theta}(\theta, \theta^\prime).
\end{aligned}
$$
Combining the last two estimates yields that
$$
\begin{aligned}
\frac{1}{2}\big||\nabla u_{\theta}(x)|^2  - |\nabla u_{\theta^\prime}(x)|^2 \big|
& \leq \frac{1}{2} \big| \nabla u_{\theta}(x) +\nabla u_{\theta^\prime}(x)\big| \big| \nabla u_{\theta}(x) -\nabla u_{\theta^\prime}(x)\big| \\ 
& \leq \Lambda_1 \rho_{\Theta}(\theta, \theta^\prime).
& 
\end{aligned}
$$
This particularly implies that $\mN(\delta, \mG^1_m, \|\cdot\|_\infty) \leq \mN(\frac{\delta}{\Lambda_1}, \Theta, \rho_\Theta)$. Then the estimate \eqref{eq:covG1} follows from Proposition \ref{prop:covtheta} with  $\delta$ replaced by $ \frac{\delta}{\Lambda_1}$.
\end{proof}

\begin{proposition}\label{prop:RadG1}
  Assume that the activation function $\phi$ satisfies Assumption \ref{assm:phi2}. Then
  $$
  R_n(\mG^1_m) \leq  \mZ(M_1,\Lambda_1,d) \cdot  \sqrt{\frac{m}{n}}.
  $$
  where 
  $M_1= \frac{1}{2} \Gamma^2 W^2(\phi^\prime_{\max})^2$ and $\Lambda_1$ is defined in \eqref{eq:Lambda21}.
\end{proposition}

\begin{proof}
Thanks to Assumption \ref{assm:phi2}, 
$$\begin{aligned}
\sup_{g\in \mG^1_m} \|g\|_{L^\infty(\Omega)} & \leq \sup_{u\in \mF_m}\frac{1}{2}\|\nabla u\|^2_{L^\infty(\Omega)}
\\
& \leq 
\frac{\Gamma^2 W^2(\phi^\prime_{\max})^2}{2}.
\end{aligned}
$$
Then the proposition follows from 
Lemma \ref{lem:coverG1}, Theorem \ref{thm:dudley} with $\delta=0$ and $M =M_1= \frac{\Gamma^2 W^2(\phi^\prime_{\max})^2}{2}$, and the simple fact that $\sqrt{a+b}\leq \sqrt{a} + \sqrt{b}$ for $a,b\geq 0$ .
\end{proof}

\subsection*{Bounding $R_n(\mG^2_m)$.}

The next lemma provides an upper bound for $\mN(\delta,\mG^2_{m},\|\cdot\|_\infty)$.
\begin{lemma}\label{lem:covG2}
Assume that  $\|f\|_{L^\infty(\Omega)}\leq F$ for some $F>0$. Assume that the activation function $\phi$ satisfies Assumption \ref{assm:phi2}. Then the  covering number  $\mN(\delta,\mG^2_{m}, \|\cdot\|_\infty)$ satisfies that 
  $$\mN(\delta,\mG^2_{m}, \|\cdot\|_\infty)\leq \mM(\delta, \Lambda_2, m, d). $$
  Here the constant $\Lambda_2$ is defined by 
  \begin{equation}\label{eq:Lambda2}
  \Lambda_2= F\bigl(1 + \phi_{\max} + 2L\Gamma\bigr).
  \end{equation}
\end{lemma}
\begin{proof}Note that a function $g_\theta \in \mG^2_m$ has the form $g_\theta = f u_\theta$. 
Given $\theta = (c, \gamma, w, t), \theta^\prime =(c^\prime, \gamma^\prime, w^\prime, t^\prime)\in \Theta$, we have 
\begin{equation}\label{eq:utheta1}
\begin{aligned}
  &|u_{\theta}(x) - u_{\theta^\prime}(x)|\leq |c-c^\prime| + \sum_{i=1}^m |\gamma_i \phi(w_i\cdot x - t_i) - \sum_{i=1}^m \gamma_i^\prime \phi(w_i^\prime\cdot x - t_i^\prime)|\\
  & \quad \leq |c-c^\prime| + \sum_{i=1}^m |\gamma_i - \gamma_i^\prime |\phi(w_i\cdot x - t_i) + \sum_{i=1}^m |\gamma_i^\prime | |\phi(w_i\cdot x - t_i)- \phi(w_i^\prime\cdot x - t_i^\prime)|.
\end{aligned}
\end{equation}
Since $\phi$ satisfies Assumption~\ref{assm:phi2}, we have that 
$
|\phi(w_i\cdot x - t_i) | \leq \phi_{\max}
$
and that 
$$
|\phi(w_i\cdot x - t_i)- \phi(w_i^\prime\cdot x - t_i^\prime)|\leq L( |w_i -w_i^\prime|_1 + |t_i- t_i^\prime|).
$$
Therefore, it follows from \eqref{eq:utheta1} that
\begin{equation}\label{eq:uthetadif}
    \begin{aligned}
  |u_{\theta}(x) - u_{\theta^\prime}(x)| & \leq |c-c^\prime| + \phi_{\max} |\gamma - \gamma^\prime|_1 \\
  & \qquad + L \Gamma ( \max_{i} |w_i -w_i^\prime|_1 +  |t-t^\prime|_\infty)\\
  & \leq \Big(1 + \phi_{\max} + 2L\Gamma\Big) \rho_{\Theta} (\theta,\theta^\prime).
  \end{aligned}
\end{equation}
This implies that $$
\|g_\theta - g_{\theta^\prime}\|_\infty \leq F\Big(1 + \phi_{\max} + 2L\Gamma\Big)\rho  = \Lambda_2 \rho_{\Theta}(\theta,\theta^\prime). $$
As a consequence, $\mN(\delta, \mG^2_{m}, \|\cdot\|_{\infty}) \leq \mN(\frac{\delta}{\Lambda_2}, \Theta, \rho_{\Theta})$. Then the lemma follows from Proposition \ref{prop:covtheta} with $\delta$ replaced by $ \frac{\delta}{\Lambda_2}$.
\end{proof}

\begin{proposition}\label{prop:RadG2}
  Assume that  $\|f\|_{L^\infty(\Omega)}\leq F$ for some $F>0$. Assume that the activation function $\phi$ is $L$-Lipschitz. Then 
  $$
  R_n(\mG^2_m) \leq  \mZ (M_2,\Lambda_2,d) \cdot  \sqrt{\frac{m}{n}},
  $$
  where 
  $M_2 = F(C+\Gamma \phi_{\max})$ and $\Lambda_2$ is defined in \eqref{eq:Lambda2}.
\end{proposition}
\begin{proof}
It follows from the definition of $\mG^2_m$ and the assumption that $\|f\|_{L^\infty(\Omega)}\leq F$, one has that 
  $
  \sup_{g\in \mG^2_m}\|g\|_{L^\infty(\Omega)} \leq M_2 = 
F(C+\Gamma \phi_{\max}).
  $
  Then the proposition is proved by 
  an application of Theorem \ref{thm:dudley} with $\delta=0, M=M_2$ and Lemma \ref{lem:covG2}.
\end{proof}

\subsection*{Bounding $R_n(\mG^3_m)$.} The  lemma below gives an upper bound for $\mN(\delta,\mG^3_{m},\|\cdot\|_\infty)$.
\begin{lemma}\label{lem:covG3}
    Assume that $\|V\|_{L^\infty(\Omega)} \leq V_{\max}$ for some $V_{\max}<\infty$.  Assume that the activation function $\phi$ satisfies Assumption \ref{assm:phi2}. Then the  covering number  $\mN(\delta,\mG^3_{m}, \|\cdot\|_\infty)$ satisfies that 
    \begin{equation}\label{eq:covG3}
\mN(\delta,\mG^3_{m}, \|\cdot\|_\infty)\leq \mM(\delta,\Lambda_3,m,d), 
\end{equation}
  where the constant $\Lambda_3$ is defined by 
\begin{equation}\label{eq:Lambda3}
      \Lambda_3  = 
  V_{\max} ( C + \Gamma \phi_{\max})\Big(1 + \phi_{\max} + 2L\Gamma\Big).
\end{equation}

\end{lemma}

\begin{proof}
    By the definition of $\mF_m$ and  Assumption \ref{assm:phi2} on $\phi$, 
    $$
    \sup_{u\in \mF_m} \|u\|_{L^\infty(\Omega)} \leq C + \Gamma \phi_{\max}.
    $$
Moreover, recall from \eqref{eq:uthetadif} that  for $\theta, \theta^\prime \in \Theta$, 
$$
|u_{\theta}(x) - u_{\theta^\prime}(x)|\leq  \Big(1 + \phi_{\max} + 2L\Gamma\Big) \rho_{\Theta} (\theta,\theta^\prime).
$$
Consequently, 
$$
\begin{aligned}
        \Big| \frac{1}{2} V(x) u_{\theta}^2(x) -  \frac{1}{2} V(x) u_{\theta^\prime}^2(x)\Big|
        & \leq \frac{1}{2} |V(x)| |u_{\theta}(x) + u_{\theta^\prime}(x)| |u_{\theta}(x) - u_{\theta^\prime}(x)|\\
        & \leq \Lambda_3 \rho_{\Theta} (\theta,\theta^\prime).
\end{aligned}
$$
The estimate \eqref{eq:covG3} follows from the same line of arguments used in the proof of Lemma~\ref{lem:covG2}.
\end{proof}

\begin{proposition}\label{prop:RadG3}
   Under the same assumption of Lemma \ref{lem:covG3}, $\mG^3_m$ satisfies that
  $$
  R_n(\mG^3_m) \leq  \mZ (M_3,\Lambda_3,d) \cdot  \sqrt{\frac{m}{n}},
  $$
  where 
  $M_3 = \frac{V_{\max}}{2} ( C + \Gamma \phi_{\max})^2$ and $\Lambda_3$ is defined in \eqref{eq:Lambda3}.
\end{proposition}

\begin{proof}
    Note that 
    $
    \sup_{u\in \mG^3_m} \|u\|_{L^\infty(\Omega)} \leq M_3=  \frac{V_{\max}}{2} ( C + \Gamma \phi_{\max})^2.  
    $
    Then  the proposition follows from
 Theorem \ref{thm:dudley} with $\delta=0, M=M_3 $ and Lemma \ref{lem:covG3}.
\end{proof}
The following corollary is a direct consequence of the Propositions \ref{prop:RadG1}-\ref{prop:RadG3}.
\begin{corollary}\label{cor:rngm}
The two sets of functions $\mG_{m,P}$ and $\mG_{m,S}$ defined in \eqref{eq:gmps} satisfy that 
$$
R_n(\mG_{m,P}) \leq ( \mZ (M_1,\Lambda_1,d) + \mZ (M_2,\Lambda_2,d) ) \cdot  \sqrt{\frac{m}{n}}
$$ and that 
$$
R_n(\mG_{m,S}) \leq \sum_{i=1}^3 \mZ (M_i,\Lambda_i,d)  \cdot  \sqrt{\frac{m}{n}}
$$ 
\end{corollary}

Considering the set of two-layer neural networks $\mF_{\sp_\tau,m}(B)$ defined in \eqref{eq:fspm} with $\tau=\sqrt{m}$, 
we define the following associated sets of functions
$$\begin{aligned}
\mG_{\sp_\tau,m,P}(B) & := \{g:\Omega \gt \R\ |\ g = \frac{1}{2} |\nabla u|^2 - f u \text{ where } u\in \mF_{\sp_\tau,m,P}(B) \},\\
\mG_{\sp_\tau,m,S}(B) & := \{g:\Omega \gt \R\ |\ g = \frac{1}{2} |\nabla u|^2 + \frac{1}{2} V|u|^2 - f u \text{ where } u\in \mF_{\sp_\tau,m,S}(B) \},\\
\mG^1_{\sp_\tau,m}(B) & := \{g:\Omega \gt \R\ |\ g = \frac{1}{2} |\nabla u|^2 \text{ where } u\in \mF_{\sp_\tau,m}(B) \},\\
\mG^2_{\sp_\tau,m}(B) & := \{g:\Omega \gt \R\ |\ g = f u \text{ where } u\in \mF_{\sp_\tau,m}(B) \}\\
\mG^3_{\sp_\tau,m}(B)   & := \Big\{g: \Omega \gt \R \ \big|\  g = \frac{1}{2} V |u|^2  \text{ where }u \in \mF_{\sp_\tau,m}(B)\Big\}.
\end{aligned}
$$

Corollary \ref{cor:rngm} allows us to bound the Rademacher complexities of  $\mG_{\sp_{\tau},m,P}(B)$ and  $\mG_{\sp_{\tau},m,S}(B)$. Indeed, from the  definition of the activation function $\sp_\tau$, we know that  $\|\sp_\tau^\prime\|_{L^\infty(\R)}\leq 1$ and   $\|\sp_{\tau}^{\prime\prime}\|_{L^\infty(\R)} \leq \tau = \sqrt{m}$, so $\sp_\tau$ satisfies Assumption \eqref{assm:phi2} with 
$$
 L =\phi^{\prime}_{\max} =  1, L^\prime = \tau = \sqrt{m} , \phi_{\max} \leq 3+ \frac{1}{\sqrt{m}} \leq 4.
$$
Note also that 
$\mF_{\sp_\tau,m,P}(B)$ coincides with the set $\mF_m$ defined in \eqref{eq:fm}  
with the following parameters
\begin{equation}\label{eq:para}
  C=2B,  \Gamma = 4B, W=1, T=1.
\end{equation}
 With the parameters above, one has that 
$$
\begin{aligned}
& M_1 = 8B,&& \qquad \Lambda_1 \leq 32B^2\sqrt{m} + 4B,\\
& M_2 \leq 18FB,&& \qquad \Lambda_2 \leq F (5 + 8B),\\
& M_3 \leq \frac{V_{\max}}{2} (18B)^2,&&\qquad  \Lambda_3 \leq 18V_{\max} B (5 + 8B).
\end{aligned}
$$ Inserting $M_i $ and $\Lambda_i, i=1,2,3$ into   \eqref{eq:Z}, one can obtain by a straightforward calculation that there exist positive constants $C_1(B,d), C_2(B,d, F)$ and $C_3(B,d, V_{\max})$, depending on the parameters $B,d, F, V_{\max}$ polynomially, such that 
$$\begin{aligned}
        & \mZ(M_1,\Lambda_1,d) \leq C_1(B, d) \sqrt{\log m},\\
        & \mZ(M_2,\Lambda_2,d) \leq C_2(B,d, F),\\
         & \mZ(M_3,\Lambda_3,d) \leq C_3(B,d, V_{\max}),\\
\end{aligned}
$$
Combining the estimates above with  Corollary \ref{cor:rngm} gives directly the  Rademacher complexity bounds for $\mG_{\sp_{\tau},m,P}(B)$ and  $\mG_{\sp_{\tau},m,S}(B)$ as summarized in the following theorem. 

\begin{theorem}\label{thm:rnGsp}
Assume that  $\|f\|_{L^\infty(\Omega)}\leq F$ and $\|V\|_{L^\infty(\Omega)}\leq V_{\max}$. Consider the sets  $\mG_{\sp_{\tau},m,P}(B)$ and  $\mG_{\sp_{\tau},m,S}(B)$ with $\tau = \sqrt{m}$. Then there exist positive constants $C_{P}(B,d,F)$ and $C_{S}(B,d,F,V_{\max})$ depending polynomially on $B,d,F,V_{\max}$  such that
$$\begin{aligned}
R_n(\mG_{\sp_{\tau},m,P}(B)) & \leq  \frac{C_{P}(B,d,F) \sqrt{m}(\sqrt{\log m}+1)}{\sqrt{n}},\\
R_n(\mG_{\sp_{\tau},m,S}(B)) & \leq  \frac{C_{S}(B,d,F,V_{\max}) \sqrt{m}(\sqrt{\log m}+1)}{\sqrt{n}}.
\end{aligned}
$$
\end{theorem}

\section{Proofs of Theorem \ref{thm:main1} and Theorem \ref{thm:main2}}\label{sec:proofmain}

With the approximation estimates for spectral Barron functions and the complexity estimates of the two-layer neural networks proved in previous sections,  we are ready to prove Theorem \ref{thm:main1} and Theorem \ref{thm:main2} which establish the a priori generalization error bounds of the DRM. 

\begin{proof}[Proof of Theorem \ref{thm:main1}] Recall that $u_{n,P}^m$ is the minimizer of  the empirical loss $\mE_{n,P}$ in the set $\mF  = \mF_{\sp_\tau, m}(B)$  with $\tau = \sqrt{m}$, where $B = \|u^\ast_P\|_{\mB(\Omega)}$.
  From the definition of $\mF_{\sp_\tau,m}(B)$, one can obtain that 
  $$
  \sup_{u\in \mF_{\sp_\tau,m}(B)} \|u\|_{L^\infty(\Omega)} \leq 14B . 
  $$
Then it follows from Theorem \ref{thm:ermbdP}, Theorem \ref{thm:rnGsp}, Theorem \ref{thm:sp} and Corollary \ref{cor:rnfm} that 
  $$\begin{aligned}
   & \bE \big[\mE_P(u_{n,P}^m) - \mE_P(u^\ast_P)\big]
    \leq 2R_n(\mG_{\sp_\tau,m,P}) + 4\sup_{u\in \mF_{\sp_\tau,m}(B)} \|u\|_{L^\infty(\Omega)} \cdot R_n(\mF_{\sp_\tau,m})\\
    &\qquad \qquad + \frac{1}{2} \inf_{u\in \mF_{\sp_{\tau,m}(B)}} \|u-u^\ast\|^2_{H^1(\Omega)}\\
    & \leq \frac{2 C_{P}(B,d,F) \sqrt{m}(\sqrt{\log m}+1)}{\sqrt{n}} + \frac{4\cdot 14\cdot 16 \cdot B^2 (\sqrt{d}+1 + \frac{\ln 2}{\sqrt{m}})}{\sqrt{n}} + \frac{B^2(6\log m +30)^2}{2m}\\
    & \leq \frac{C_1 \sqrt{m} (\sqrt{\log m}+1)}{\sqrt{n}} + \frac{C_2(\log m+1)^2}{m},
  \end{aligned}
  $$
  where the constant $C_1$ depends polynomially on $B,d$ and $F$ and $C_2$  depends only quadratically on $B$. 
\end{proof}

\begin{proof}[Proof of Theorem \ref{thm:main2}]
The proof is almost identical to the proof of Theorem \ref{thm:main1} and follows directly from Theorem \ref{thm:ermbdS}, Theorem \ref{thm:rnGsp}, Theorem \ref{thm:sp} and Corollary \ref{cor:rnfm}. Hence we  omit the details. 
\end{proof}

\section{Solution theory of Poisson and  static Schr\"odinger Equations in spectral Barron Spaces}
In Theorems \ref{thm:main1} and \ref{thm:main2}, we have established the generalization error bounds of the DRM for the Poisson equation and static Schr\"odinger equation  under the assumption that the exact solutions lie in the spectral Barron space $\mB(\Omega)$. This section aims to justify such assumption by proving complexity estimates of solutions in the spectral Barron space as shown in Theorem \ref{thm:complexpoisson}  and Theorem \ref{thm:complexschr}. This can be viewed as regularity analysis of high dimensional PDEs in the spectral Barron space.

\subsection{Proof of  Theorem \ref{thm:complexpoisson}} 
\label{sec:complexpoisson}
Suppose that $f = \sum_{k\in \N_0^d} \hat{f}_k \Phi_k$ and that $f$ has vanishing mean value on $\Omega$ so that $\hat{f}_0  = 0$. Let $\hat{u}_k$ be the cosine coefficients of the solution $u^\ast_P$ of the Neumann problem for Poisson equation. By testing $\Phi_k$ on both sides of the Poisson equation and by taking account of the Neumann boundary condition, one obtains that  
$$
\begin{aligned}
\hat{u}_0 &= 0,\\
\hat{u}_k &= -\frac{1}{\pi^2|k|^2}\hat{f}_k.
\end{aligned}
$$
As a result,
$$\begin{aligned}
\|u^\ast_P\|_{\mB^{s+2}(\Omega)} & = \sum_{k\in \N_0^d\setminus \{\mathbf{0}\}} (1 + \pi^{s+2}|k|_1^{s+2}) |\hat{u}_k| =
\sum_{k\in \N_0^d\setminus \{\mathbf{0}\}} \frac{(1 + \pi^{s+2}|k|_1^{s+2})}{\pi^2|k|^2} |\hat{f}_k| \\
& \leq d \sum_{k\in \N_0^d\setminus \{\mathbf{0}\}} (1 + \pi^{s}|k|^{s})  |\hat{f}_k| = d\|f\|_{\mB^s(\Omega)},
\end{aligned}
$$
where we have used $|k|_1^2 \leq d |k|^2$ in the  inequality above. 
This finishes the proof.

\subsection{Proof of Theorem \ref{thm:complexschr}}
\label{sec:proofcomplexS}
First under the assumption of Theorem \ref{thm:complexschr}, there exists a unique solution $u\in H^1(\Omega)$ to \eqref{eq:schrneumann}. Moreover, 
\begin{equation}\label{eq:wellposed}
    \|\nabla u\|^2_{L^2(\Omega)} + V_{\min} \|u\|^2_{L^2(\Omega)} \leq \|f\|_{L^2(\Omega)}  \|u\|_{L^2(\Omega)}. 
\end{equation}
Our goal is to show that $u\in \mB^{s+2}(\Omega)$. To this end, let us first derive an  operator equation that is equivalent to the original Schr\"odinger problem \eqref{eq:schrneumann}. To do this, multiplying $\Phi_k$ on both sides of the static Schr\"odinger  equation and then integrating yields the following equivalent linear system on $\hat{u}$:
\begin{equation}\label{eq:schr2}
-|\pi|^2 |k|^2 \hat{u}_k + \widehat{(V u)}_k = \hat{f}_k, \quad  k  \in \N_0^d.
\end{equation}
Let us first consider \eqref{eq:schr2} with $k= \mathbf{0}$. Thanks to  Corollary \ref{cor:hatuv}, 
$$\begin{aligned}
\widehat{(Vu)}_0 & = \frac{1}{\beta_0} \Big(\sum_{m\in \Z^d}\beta_m^2 \hat{u}_{|m|} \hat{V}_{|m|}\Big)
& = \hat{u}_0 \hat{V}_0 + \Big(\sum_{m\in \Z^d\setminus \{\mathbf{0}\}}\beta_m^2 \hat{u}_{|m|} \hat{V}_{|m|}\Big),
\end{aligned}
$$
where we have also used the fact that $\beta_0 = 1$. Consequently,  equation \eqref{eq:schr2} with $k= \mathbf{0}$ becomes 
$$
\hat{u}_0 \hat{V}_0 + \sum_{m\in \Z^d\setminus \{\mathbf{0}\}}\beta_m^2 \hat{u}_{|m|} \hat{V}_{|m|} = \hat{f}_0.
$$
For $k\neq \mathbf{0}$, using again  Corollary \ref{cor:hatuv},  equation \eqref{eq:schr2} can be written as  
$$
-|\pi|^2 |k|^2 \hat{u}_k + \frac{1}{\beta_k} \Big(\sum_{m\in \Z^d}\beta_m \hat{u}_{|m|} \beta_{m-k} \hat{V}_{|m-k|}\Big) = \hat{f}_k, \quad  k\in \N^d\setminus \{\mathbf{0}\}. 
$$
Recall that a function $u\in \mB^s(\Omega)$ is equivalent to that $\hat{u}_k$ belongs to the weighted $\ell^1$ space $\ell^1_{W_s}(\N_0^d)$  with the weight $W_s(k) = 1 + \pi^s|k|_1^s$.  We would like to rewrite the above equations as an operator equation on the space $\ell^1_{W_s}(\N_0^d)$. For doing this, let us define some useful operators. Define the operator $\mathbb{M}: \hat{u} \mapsto \mathbb{M} \hat{u}$ by 
$$
(\mathbb{M} \hat{u} )_k = \begin{cases}
\hat{V}_0  \hat{u}_0  & \text{ if } k = \mathbf{0},\\
-|\pi|^2 |k|^2 \hat{u}_k & \text{ otherwise}. 
\end{cases}
$$
Define the operator $\mathbb{V} : \hat{u} \mapsto \mathbb{V} \hat{u}$ by  
$$
 (\mathbb{V} \hat{u})_k = \begin{cases}
 	\sum_{m\in \Z^d\setminus \{\mathbf{0}\}}\beta_m^2 \hat{u}_{|m|} \hat{V}_{|m|} & \text{ if } k = \mathbf{0},\\
 	\frac{1}{\beta_k} \Big(\sum_{m\in \Z^d}\beta_m \hat{u}_{|m|} \beta_{m-k} \hat{V}_{|m-k|}\Big) & \text{ otherwise}. 
 \end{cases}
$$
With those operators, the system \eqref{eq:schr2} can be reformulated as the operator equation 
\begin{equation}\label{eq:ope1}
    (\mathbb{M} + \mathbb{V})\hat{u} = \hat{f}.
\end{equation}
Since $V(x)\geq V_{\min} > 0$ for every $x$, we have $\hat{V}_0 >0$. As a direct consequence, the diagonal operator $\mathbb{M}$ is invertible. Therefore the operator equation \eqref{eq:ope1} is equivalent to 
\begin{equation}\label{eq:ope2}
(\mathbb{I} + \mathbb{M}^{-1} \mathbb{V}) \hat{u} = \mathbb{M}^{-1} \hat{f}.
\end{equation}

In order to show that  $u\in \mB^{s+2}(\Omega)$, it suffices to show that the equation \eqref{eq:ope1} or \eqref{eq:ope2} has a unique solution $\hat{u}\in \ell^1_{W_s}(\N_0^d)$. Indeed, if  $\hat{u}\in \ell^1_{W_s}(\N_0^d)$, then it follows from \eqref{eq:ope1} and the boundedness of $\mathbb{V}$ on $\ell^1_{W_s}(\N_0^d)$ (see \eqref{eq:bdM} in the proof of Lemma \ref{lem:compactMV} below)  that 
    \begin{equation}\begin{aligned}\label{eq:Mu}
    \|\mathbb{M} \hat{u}\|_{ \ell^1_{W_s}(\N_0^d)}  & \leq  \|\mathbb{V} \hat{u}\|_{ \ell^1_{W_s}(\N_0^d)} +  \|\hat{f}\|_{\ell^1_{W_s}(\N_0^d)} \\
    & \leq C(d, V) \|\hat{u}\|_{ \ell^1_{W_s}(\N_0^d)} +  \|\hat{f}\|_{\ell^1_{W_s}(\N_0^d)} .
    \end{aligned}
    \end{equation}
   Moreover,  this combined with the positivity of $\hat{V}_0$ implies that 
    \begin{equation}\begin{aligned}\label{eq:uB}
    \|u\|_{\mB^{s+2}(\Omega)} & = \sum_{k\in \N_0^d} (1 + \pi^{s+2}|k|_1^{s+2}) |\hat{u}_k| \\
    & = \frac{1}{\hat{V}_0}\cdot \hat{V}_0 |\hat{u}_0|  + \sum_{k\in \N^d_0\setminus \{\mathbf{0}\}} \frac{1 + \pi^{s+2}|k|_1^{s+2}}{\pi^2 |k|^2} \cdot \pi^2 |k|^2 |\hat{u}_k| \\
    & \leq \max \Big\{\frac{1}{\hat{V}_0}, \Big(\frac{1}{\pi^2}+d\Big)\Big\} \|\mathbb{M} \hat{u}\|_{ \ell^1_{W_s}(\N_0^d)} \\
    & \leq C_1(d,V) (\|\hat{u}\|_{ \ell^1_{W_s}(\N_0^d)} +  \|\hat{f}\|_{\ell^1_{W_s}(\N_0^d)})
    \end{aligned}
     \end{equation}
     for some $C_1(d, V) >0$.

Next, we claim that equation \eqref{eq:ope2} has a unique solution $\hat{u}\in \ell^1_{W_s}(\N_0^d)$ and that there exists a constant $C_2>0$ such that 
\begin{equation}\label{eq:bduhat}
    \|\hat{u}\|_{\ell^1_{W_s}(\N_0^d)} \leq C_2 \|\hat{f}\|_{\ell^1_{W_s}(\N_0^d)}.
\end{equation}
To see this, observe that owing  to the compactness of $\mathbb{M}^{-1}\mathbb{V}$ as shown in Lemma \ref{lem:compactMV}, the operator equation $\mathbb{I} + \mathbb{M}^{-1}\mathbb{V}$ is a Fredholm operator on $\ell^1_{W_s}(\N_0^d)$. By the celebrated Fredholm alternative theorem (see e.g., \cite{fredholm1903class} and \cite[VII 10.7]{Conway90}), the operator  $\mathbb{I} + \mathbb{M}^{-1}\mathbb{V}$ has a bounded inverse $(\mathbb{I} + \mathbb{M}^{-1}\mathbb{V})^{-1}$ if and only if $(\mathbb{I} + \mathbb{M}^{-1}\mathbb{V})\hat{u} = 0$ has a trivial solution. Therefore to obtain the bound \eqref{eq:bduhat}, it suffices to show that  $(\mathbb{I} + \mathbb{M}^{-1}\mathbb{V})\hat{u} = 0$ implies $\hat{u}=0$. By the equivalence between the Schr\"odinger problem \eqref{eq:schrneumann} and \eqref{eq:ope2}, we only need to show that the only solution of \eqref{eq:schrneumann} is zero. Notice that the latter is a direct consequence of \eqref{eq:wellposed} and thus this finishes the proof of that the Schr\"odinger problem \eqref{eq:schrneumann} has a unique solution in $\mB(\Omega)$. 
Finally, the stability estimate \eqref{eq:complexschr} follows by combining \eqref{eq:uB} and \eqref{eq:bduhat}.

\begin{lemma}\label{lem:compactMV}
    Assume that $V\in \mB^s(\Omega)$ with $V(x)\geq V_{\min} >0$ for every $x\in \Omega$. Then the operator $ \mathbb{M}^{-1} \mathbb{V}$ is compact on $\ell^1_{W_s}(\N_0^d)$.
\end{lemma}
\begin{proof}
Since $\mathbb{M}^{-1}$ is a multiplication operator on $\ell^1_{W_s}(\N_0^d)$ with the diagonal entries converging to zero, it follows from Lemma \ref{lem:compactT} that $\mathbb{M}^{-1}$ is compact on $\ell^1_{W_s}(\N_0^d)$. Therefore to show the compactness of $ \mathbb{M}^{-1} \mathbb{V}$, it is sufficient to show that the operator $\mathbb{V}$ is bounded on $\ell^1_{W_s}(\N_0^d)$. To see this, note that by definition $\beta_k = 2^{\mathbf{1}_k - \sum_{i=1}^d \mathbf{1}_{k_i\neq 0}} \in [2^{1-d}, 2]$. In addition, 
since $V\in \mB^0(\Omega)$, using Corollary~\ref{cor:hatuv}, one has that 
\begin{equation}\begin{aligned}\label{eq:bdM}
& \|\mathbb{V} \hat{u}\|_{\ell^1_{W_s}(\N_0^d)} = \Big|\sum_{m\in \Z^d\setminus \{\mathbf{0}\}}\beta_m^2 \hat{u}_{|m|} \hat{V}_{|m|}\Big|
+ \sum_{k\in \N^d} \frac{1}{\beta_k} \Big|\sum_{m\in \Z^d}\beta_m \hat{u}_{|m|} \beta_{m-k} \hat{V}_{|m-k|}\Big| (1 + \pi^s |k|_1^s)\\
& \leq 4 \sum_{m\in \Z^d\setminus \{\mathbf{0}\}} |\hat{u}_{|m|}| \sum_{m\in \Z^d\setminus \{\mathbf{0}\}} |\hat{V}_{|m|}| + 2^{d+1} \sum_{m\in \Z^d} \sum_{k\in \N^d}  |\hat{u}_{|m|}| |\hat{V}_{|m-k|}| \big(1 + |\pi|^s C_s(|m-k|_1^s + |m|_1^s)\big)
\\
& \leq 2^{2d+2} \|\hat{u}\|_{\ell^1(\N_0^d)} \|\hat{V}\|_{\ell^1(\N_0^d)} + 2^{2d+1} \max(1,C_s)\cdot \big(\|\hat{u}\|_{\ell^1(\N_0^d)} \|\hat{V}\|_{\ell^1_{W_s}(\N_0^d)} + \|\hat{u}\|_{\ell^1_{W_s}(\N_0^d)} \|\hat{V}\|_{\ell^1(\N_0^d)}\big)\\
& \leq 2^{2d+3}  \max(1,C_s)\cdot   \|\hat{V}\|_{\ell^1_{W_s}(\N_0^d)}\|\hat{u}\|_{\ell^1_{W_s}(\N_0^d)} \\
& = 2^{2d+3}  \max(1,C_s)\cdot   \|V\|_{\mB^s(\Omega)} \|\hat{u}\|_{\ell^1_{W_s}(\N_0^d)},
\end{aligned}
\end{equation}
where in the first inequality above we used the elementary inequality $|a+b|^s \leq C_s (|a|^s +|b|^s)$ for some constant $C_s>0$ and in the second inequality  we  used the fact that $\sum_{m\in \Z^d} |\hat{u}_{|m|}| \leq 2^d  \|\hat{u}\|_{\ell^1(\N_0^d)}\leq 2^d  \|\hat{u}\|_{\ell^1_{W_s}(\N_0^d)}$.
\end{proof}

\begin{lemma} \label{lem:compactT}
Suppose that $\mathbb{T}$ is a multiplication operator on $\ell^1_{W_s}(\N_0^d)$ defined by for $u= (u_k)_{k\in \N_0^d}$ that $(\mathbb{T}u)_k= \lambda_k u_k$ with $\lambda_k \gt 0$ as $\|k\|_2 \gt \infty$. Then $\mathbb{T}: \ell^1_{W_s}(\N_0^d) \gt \ell^1_{W_s}(\N_0^d)$ is compact. 
    
\end{lemma}

\begin{proof}
  It suffices to show that the image of the unit ball in $\ell^1_{W_s}(\N_0^d)$ under the map $\mathbb{T}$ is totally bounded. To this end, given any fixed $\eps>0$, let $K_0\in \N$ be such that $|\lambda_k| \leq \eps$ if $\|k\|_2> K_0$. Denote by $\mI_{0}: \{k\in \N_0^d: \|k\|_2\leq K_0\}$ and let  $d_0$ be the cardinality of the index set $\mI_{0}$.  Note that the ball in $\R^{d_0}$ of radius $\max_k\{ |\lambda_k|: k\in \mI_0\}$ with respect to the weighted $1$-norm $\|v\|_{\ell^1_{W_{s}}} = \sum_{k\in \mI_0} |v_k| W_s(k)$ is precompact, so it can be covered by the union of $n$ $\eps$-balls with  centers $\{v_1, \cdots, v_n\}$ where $v_i\in \R^{d_0}$. We now claim that the image of the unit ball in $\ell^1_{W_s}(\N_0^d)$ under $\mathbb{T}$ is covered by $n$ $2\eps$-balls with  centers $\{(v_1, \mathbf{0}), \cdots, (v_n,\mathbf{0})\}$. In fact, for $u \in \ell^1_{W_s}(\N_0^d)$ with $\sum_{k\in \N_0^d} |u_k|W_s(k) \leq 1$, one has 
    $$
    \mathbb{T} u  = \Big((\lambda_k u_k)_{k\in \mI_0}, \mathbf{0}\Big) + \Big(\mathbf{0}, (\lambda_k u_k)_{k\notin \mI_0} \Big). 
    $$
    Suppose that $v_{i^\ast}$ is the closest center of $\{v_1,\cdots, v_n\}$ to the vector $\big((\lambda_k u_k)_{k\in \mI_0}\big)$. Then 
    $$\begin{aligned}
    \|T u - (v_{i^\ast}, \mathbf{0})\|_{\ell^1_{W_s}(\N_0^d)} & = \sum_{k\in \mI_0} |(v_{i^\ast})_k  - (\lambda_k u_k)| W_s(k) + \Big\|\Big(\mathbf{0}, (\lambda_k u_k)_{k\notin \mI_0} \Big)\Big\|_{\ell^1_{W_s}(\N_0^d)}\\
    & \leq \eps + \eps \Big\|\Big(\mathbf{0}, ( u_k)_{k\notin \mI_0} \Big)\Big\|_{\ell^1_{W_s}(\N_0^d)} \leq 2 \eps. 
    \end{aligned}$$
    This finishes the proof. 
\end{proof}

\appendix

\section{Proof of Proposition \ref{prop:pde}}\label{sec:proppde}

\subsection{Proof of Proposition \ref{prop:pde}-(i)}
First, it is well known that the problem \eqref{eq:Poisson} has a unique weak solution $u^\ast_P \in H^1_\diamond(\Omega) = \{u\in H^1(\Omega): \int_{\Omega} u dx = 0\}$, i.e. 
\begin{equation}\label{eq:weakform2}
    a(u, v) =: \int_{\Omega} \nabla u \cdot \nabla v  =  F(v):= \int_{\Omega} f v dx \text{ for every } v\in H^1_\diamond(\Omega).
\end{equation}
Moreover, the solution $u^\ast_P$ satisfies that 
$$
u^\ast_P = \argmin_{u\in H^1_\diamond(\Omega)   } \Big\{\frac{1}{2} \int_{\Omega} |\nabla u|^2 dx - \int_{\Omega} f udx \Big\}.
$$

Due to the mean-zero constraint of the space $H^1_\diamond(\Omega) $, the variational formulation above is inconvenient to be adopted as the loss function for training a neural network solution. To tackle this issue, we consider instead the following modified Poisson problem:
\begin{equation}\label{eq:neumann2}
\begin{aligned}
   -\Delta u + \lambda \int_{\Omega} u dx & = f \text{ on } \Omega,\\ 
     \frac{\partial}{\partial \nu} u & =  0 \text{ on } \partial \Omega. 
\end{aligned}
\end{equation}
Here $\lambda>0$ is a fixed constant. By the  Lax-Milgram theorem the problem \eqref{eq:neumann2} has a unique weak solution $u^\ast_{\lambda,P}$, which solves 
\begin{equation}\label{eq:weakform3}
    a_\lambda(u^\ast_{\lambda,P}, v):= \int_{\Omega} \nabla u \cdot \nabla v dx + \lambda \int_{\Omega} u dx \int_{\Omega} v dx   = F(v) \text{ for every } v\in H^1(\Omega).
\end{equation}
It is clear that $u^\ast_{\lambda,P}$ is the solution of 
the variational problem
\begin{equation}\label{eq:variationlambda}
 \argmin_{u\in H^1(\Omega)} \Big\{\frac{1}{2} \int_{\Omega} |\nabla u|^2 dx +  \frac{\lambda}{2} \Big(\int_{\Omega} u dx\Big)^2 - \int_{\Omega} f u dx \Big\}.
\end{equation}
Furthermore, the lemma below shows that the weak solutions of \eqref{eq:neumann2} are independent of $\lambda$ and they all coincides with $u^\ast_P$. 

\begin{lemma} \label{lem:equiv0}
Assume that $\lambda>0$.
Let $u^\ast_P$ and $u^\ast_{\lambda,P}$ be the weak solution of \eqref{eq:Poisson} and \eqref{eq:neumann2} respectively with $f\in L^2(\Omega)$ satisfying $\int_{\Omega} f dx =0$.   Then we have that $u^\ast_{\lambda,P}  = u^\ast_P$. 
\end{lemma}
\begin{proof}
We only need to show that $u^\ast_{\lambda,P}$  satisfies the weak formulation \eqref{eq:weakform2}. In fact, since $u^\ast_{\lambda,P}$ satisfies \eqref{eq:weakform3}, by setting $v = 1$ we obtain that 
$$
\lambda \int_{\Omega} u dx  = \int_{\Omega} f d x  = 0. 
$$
This immediately implies that $a_\lambda(u^\ast_{\lambda,P},v) = a(u^\ast_{\lambda,P},v)$ and hence $u^\ast_{\lambda,P}$ satisfies \eqref{eq:weakform2}.
\end{proof}
Since the solution to \eqref{eq:neumann2} is invariant for all $\lambda>0$, for simplicity we  set $\lambda=1$ in \eqref{eq:variationlambda} and this proves \eqref{eq:poissonneumann2}, i.e. 
\begin{equation}\label{eq:vrp}
  u^\ast_P=  \argmin_{u\in H^1(\Omega)} \mE_P(u) = \argmin_{u\in H^1(\Omega)} \Big\{\frac{1}{2} \int_{\Omega} |\nabla u|^2 dx  - \int_{\Omega} f u dx +  \frac{1}{2} \Big(\int_{\Omega} u dx\Big)^2 \Big\}.
\end{equation}
Finally we prove that $u^\ast_P$ satisfies the estimate \eqref{eq:equiv1}. To see this, we first state a useful lemma which computes the energy excess $\mE (u) - \mE(u^\ast_P) $ with any $u\in H^1(\Omega)$.

\begin{lemma}\label{lem:neumann}
Let $u^\ast_P$ be the minimizer of  $\mE_P$ or equivalently the weak solution of the Poisson problem \eqref{eq:neumann2}. Then for any $u\in H^1(\Omega)$, it holds that 
$$
\mE_P (u) - \mE_P(u^\ast_P) = 
 \frac{1}{2} \int_\Omega |\nabla u - \nabla u^\ast_P|^2 dx + \frac{1}{2}\Big( \int_{\Omega} u^\ast_P - u\ dx\Big)^2.
$$
\end{lemma}

\begin{proof}
It follows from Green's formula and the fact that $u^\ast_P \in H^1_{\diamond}(\Omega)$ that 
$$\begin{aligned}
    \mE (u^\ast_P) & = \int_\Omega \frac{1}{2} |\nabla u^\ast_P|^2 - fu^\ast_P dx + \underbrace{\frac{1}{2}\Big( \int_\Omega u^\ast_P d x\Big)^2}_{=0} \\
    & = \int_\Omega \frac{1}{2} |\nabla u^\ast_P|^2 + \Delta u^\ast_P u^\ast_P dx  \\
    & = -\frac{1}{2} \int_\Omega |\nabla u^\ast_P|^2 dx   . 
\end{aligned}
$$
Then for any $u\in  H^1(\Omega)$, applying Green's formula again yields  
\begin{align*}
\mE (u) - \mE(u^\ast_P) 
& = \frac{1}{2}\int_\Omega  |\nabla u|^2 dx - \int_\Omega fu dx + \frac{1}{2} \Big(\int_\Omega u dx \Big)^2 +  \frac{1}{2} \int_\Omega |\nabla u^\ast_P|^2 dx \\
& = \frac{1}{2} \int_\Omega  |\nabla u|^2  dx + \int_\Omega \Delta u^\ast_P u dx +  \frac{1}{2} \Big(\int_\Omega u dx \Big)^2 +  \frac{1}{2} \int_\Omega |\nabla u^\ast_P|^2 dx\\
& = \frac{1}{2} \int_\Omega |\nabla u - \nabla u^\ast_P|^2 dx + \frac{1}{2}\Big( \int_{\Omega} (u^\ast_P - u)\ dx\Big)^2. \qedhere
\end{align*}
\end{proof}
Now recall that $C_P>0$ is the  Poincar\'e constant such that for any $v\in H^1(\Omega)$,
$$
\Big\|v- \int_{\Omega} v dx  \Big\|_{L^2(\Omega)}^2 \leq C_P \|\nabla v\|_{L^2(\Omega)}^2 .
$$
As a result,  
$$\begin{aligned}
\|v\|_{H^1(\Omega)}^2  & = \|\nabla v\|_{L^2(\Omega)}^2 +  \| v\|_{L^2(\Omega)}^2 \\
& \leq 
\|\nabla v\|_{L^2(\Omega)}^2
+  2\Big\| v- \int_{\Omega} v\Big\|_{L^2(\Omega)}^2 +  2\Big| \int_{\Omega} vdx\Big|^2\\
& \leq (2C_P+1)\|\nabla v\|_{L^2(\Omega)}^2 +2\Big| \int_{\Omega} vdx\Big|^2.
\end{aligned}
$$
Therefore, an application of the last inequality with $v = u-u^\ast_P$ and Lemma \ref{lem:neumann} yields that 
$$
\|u-u^\ast_P\|_{H^1(\Omega)}^2 \leq 2 \max\{2C_P +1, 2\} (\mE(u)-\mE(u^\ast_P)). 
$$
On the other hand, it follows from Lemma \ref{lem:neumann}  that
$$
\mE(u)-\mE(u^\ast_P) \leq \frac{1}{2}\|u-u^\ast_P\|_{H^1(\Omega)}^2.
$$
Combining the last two estimates leads to \eqref{eq:equiv1} and hence finishes the proof of Proposition \ref{prop:pde}-(i).

\subsection{Proof of Proposition \ref{prop:pde}-(ii)}
First the standard Lax-Milgram theorem implies that  the static Schr\"odinger equation has a unique weak solution $u^\ast_S$. Moreover, it is not hard to verify that $u^\ast_S$ solves the equivalent variational problem \eqref{eq:schrneumann2}, i.e. 
   \begin{equation*}
    u^\ast_S =  \argmin_{u\in H^1(\Omega)} \mE_S(u) =  \argmin_{u\in H^1(\Omega)}  \Big\{\frac{1}{2} \int_{\Omega} |\nabla u|^2  +   V |u|^2 \ dx- \int_{\Omega} f u dx \Big\},
\end{equation*}
Finally we prove that $u^\ast_S$ satisfies the estimate \eqref{eq:equiv2}. For this, we first claim that for any $u\in H^1(\Omega)$, 
\begin{equation}\label{eq:energyexcessS}
\mE_S (u) - \mE_S(u^\ast_S) = 
 \frac{1}{2} \int_\Omega |\nabla u - \nabla u^\ast_S|^2 dx + \frac{1}{2} \int_{\Omega} V( u^\ast_S - u)^2\ dx.
\end{equation}
In fact, using Green's formula, one has that 
$$\begin{aligned}
    \mE_S (u^\ast_S) & = \int_\Omega \frac{1}{2} |\nabla u^\ast_S|^2  + \frac{1}{2} V |u^\ast_S|^2 - fu^\ast dx  \\
    & = \int_\Omega \frac{1}{2} |\nabla u^\ast_S|^2  + \frac{1}{2} V |u^\ast_S|^2  + (\Delta u^\ast_S - V u^\ast_S) u^\ast dx  \\
    & = -\frac{1}{2} \int_\Omega |\nabla u^\ast_S|^2  + V |u^\ast|^2 dx   . 
\end{aligned}
$$
Then for any $u\in  H^1(\Omega)$, applying Green's formula again yields  
\begin{align*}
\mE_S (u) - \mE_S(u^\ast_S) 
& = \frac{1}{2} \int_{\Omega} |\nabla u|^2  +   V |u|^2  dx- \int_{\Omega} f u dx + \frac{1}{2} \int_\Omega |\nabla u^\ast_S|^2  + V |u^\ast_S|^2 dx  \\
& = \frac{1}{2} \int_\Omega  |\nabla u|^2  +   V |u|^2 dx + \int_\Omega  (\Delta u^\ast_S - V u^\ast_S) u dx +  \frac{1}{2} \int_\Omega |\nabla u^\ast_S|^2  + V |u^\ast_S|^2 dx\\
& = \frac{1}{2} \int_\Omega |\nabla u - \nabla u^\ast_S|^2 dx + \frac{1}{2} \int_{\Omega} V \big(u^\ast_S - u\big)^2 dx. 
\end{align*}
The estimate \eqref{eq:equiv2} follows directly from the identity \eqref{eq:energyexcessS} and the assumption that $0< V_{\min} \leq V(x)\leq V_{\max}$. This completes the proof. 

\section{Some useful facts on cosine series and convolution}

Assume that $u\in L^1(\Omega)$ admits the cosine series expansion
$$ 
u(x) = \sum_{k\in \N_0^d} \hat{u}_k \Phi_k(x),
$$
where $\{\hat{u}_k\}_{k\in \N_0^d}$ are the cosine expansion coefficients, i.e. 
\begin{equation}\label{eq:uhatk}
 \hat{u}_k = \frac{\int_{\Omega} u(x) \Phi_k(x)dx}{\int_{\Omega} \Phi_k^2(x)dx } =  \frac{\int_{\Omega} u(x) \Phi_k(x)dx}{2^{-\sum_{i=1}^d \mathbf{1}_{k_i\neq 0}}}.
\end{equation}
Let $\Omega_e := [-1,1]^d$ and  define the even extension of $u_e$ of a function $u$ by 
$$
u_e(x) = u_e(x_1,\cdots, u_d) =  u(|x_1|,\cdots, |x_d|), x\in \Omega_e.
$$
Let $\tilde{u}_k$ be the Fourier coefficients of $u_e$. Since $u_e$ is real and even, one has that 
$$
u_e = \sum_{k\in \Z^d} \tilde{u}_k \cos(\pi k\cdot x),  
$$
where 
\begin{equation}\label{eq:utildek}
	\tilde{u}_k = \frac{\int_{\Omega_e} u_e(x) \cos(\pi k\cdot x) dx}{\int_{\Omega_e}  \cos^2(\pi k\cdot x) dx} = \frac{1}{2^{d - \mathbf{1}_{k\neq \mathbf{0}}}}\int_{\Omega_e} u_e(x) \cos(\pi k\cdot x) dx.
\end{equation}
By abuse of notation, we use $|k|$ to stand for the vector $(|k_1|, |k_2,|,\cdots, |k_d|)$.

\begin{lemma}\label{lem:betak}
	For every $k\in \Z^d$, it holds that $\tilde{u}_k = \beta_k \hat{u}_{|k|}$ where $\beta_k = 2^{\mathbf{1}_{k\neq \mathbf{0}}-\sum_{i=1}^d \mathbf{1}_{k_i\neq 0}}$. 
\end{lemma}

\begin{proof}
	First thanks to Lemma \ref{lem:cos} and the evenness of cosine, 
$$\begin{aligned}
 \int_{\Omega_e} u_e(x) \cos(\pi k \cdot x)dx & = \int_{\Omega_e} u_e(x) \cos\Big(\pi\Big(\sum_{i=1}^{d-1} k_i x_i\Big)\Big) \cos(\pi k_d x_d) dx \\
& - \underbrace{\int_{\Omega_e} u_e(x) \sin\Big(\pi\Big(\sum_{i=1}^{d-1} k_i x_i\Big)\Big) \sin(\pi k_d x_d) dx}_{=0}\\
& =  \int_{\Omega_e} u_e(x) \cos\Big(\pi\Big(\sum_{i=1}^{d-2} k_i x_i\Big)\Big) \cos(\pi k_{d-1} x_{d-1}) \cos(\pi k_d x_d) dx \\
& - \underbrace{\int_{\Omega_e} u_e(x) \sin \Big(\pi\Big(\sum_{i=1}^{d-2} k_i x_i\Big)\Big) \sin (\pi k_{d-1} x_{d-1}) \cos(\pi k_d x_d) dx}_{=0}\\
& = \cdots \\
& =  \int_{\Omega_e} u_e(x)  \prod_{i=1}^d \cos(\pi k_i x_i ) dx \\
& = 2^d \int_{\Omega} u(x) \Phi_k(x) dx.
\end{aligned}
$$
In addition, since $\Phi_k  = \Phi_{|k|}$ for any $k\in \Z^d$, the lemma follows from the equation above, \eqref{eq:uhatk} and \eqref{eq:utildek}.   
\end{proof}

The next lemma shows that the Fourier coefficients of the product of two functions $u$ and $v$ are the discrete convolution of their Fourier coefficients. Recall that $\{\tilde{u}_k\}_{k\in \Z^d}$ denote the Fourier coefficients of the even functions $u_e$. 

\begin{lemma}\label{lem:convuv}
	Let $w_e = u_e v_e$.  Then $\tilde{w}_k = \sum_{m\in \Z^d} \tilde{u}_m \tilde{v}_{k-m}$.
\end{lemma}

\begin{proof}
	By definition,
$
u_e(x) = \sum_{m\in \Z^d} \tilde{u}_m \cos(\pi m \cdot x)
$ and $
v_e(x) = \sum_{n\in \Z^d} \tilde{v}_n \cos(\pi n \cdot x)
$
Thanks to  the fact that 
$$
\int_{\Omega_e} \cos(\pi \ell\cdot x) \cos(\pi k \cdot x) = 2^{d - \mathbf{1}_{k\neq \mathbf{0}}} \delta_{\ell}(k),
$$
one obtains that 
	 $$
	 \begin{aligned}
	 		\tilde{w}_k  &=  \frac{1}{2^{d - \mathbf{1}_{k\neq \mathbf{0}}}}\int_{\Omega_e} u_e(x) v_e(x) \cos(\pi k\cdot x) dx\\
	 		& = \frac{1}{2^{d - \mathbf{1}_{k\neq \mathbf{0}}}}
	 		\sum_{m\in \Z^d} \sum_{n\in \Z^d} \tilde{u}_m \tilde{v}_n
	 		\int_{\Omega_e} 
	 		 \cos(\pi m \cdot x)
	 		  \cos(\pi n \cdot x)\cos(\pi k\cdot x) dx\\
	 		  & = \frac{1}{2^{d - \mathbf{1}_{k\neq \mathbf{0}}}}
	 		\sum_{m\in \Z^d} \sum_{n\in \Z^d} \tilde{u}_m \tilde{v}_n
	 		\int_{\Omega_e} \frac{1}{2} \Big[\cos(\pi(m+n)\cdot x) + \cos (\pi(m-n)\cdot x)\Big] \cos(\pi k\cdot x) dx\\
	 		& = \frac{1}{2} \sum_{m\in \Z^d} \tilde{u}_m (\tilde{v}_{k-m} + \tilde{v}_{m-k})\\
	 		& = \sum_{m\in \Z^d} \tilde{u}_m \tilde{v}_{k-m},
	 \end{aligned}
	$$
	where we have also used that $\tilde{v}_k = \tilde{v}_{-k}$ for any $k$. 
\end{proof}

\begin{corollary}\label{cor:hatuv}
For any $k\in \N^d$, 
$$
\widehat{(u v)}_k =  \frac{1}{\beta_k} \sum_{m\in \Z^d}\beta_m \hat{u}_{|m|} \beta_{m-k} \hat{v}_{|m-k|}.
$$	
\end{corollary}
\begin{proof}
Thanks to Lemma \ref{lem:betak} and Lemma \ref{lem:convuv}, 
\begin{equation*}
\widehat{(u v)}_k = \frac{1}{\beta_k} \widetilde{(u v)}_k  =  \frac{1}{\beta_k} (\tilde{u} \ast \tilde{v} )_k =  \frac{1}{\beta_k} \sum_{m\in \Z^d}\beta_m \hat{u}_{|m|} \beta_{m-k} \hat{v}_{|m-k|}. \qedhere
\end{equation*}
\end{proof}

\bibliographystyle{plain}
\bibliography{references}
\end{document}